\documentclass[11pt,oneside,reqno]{amsart}

\usepackage{graphicx}
\usepackage{pict2e}
\usepackage{amssymb}
\usepackage{amsthm}                % better theorem environments
\usepackage[margin=1in]{geometry}  % set the margins to 1in on all sides
\usepackage{graphics}
\usepackage{epstopdf}
\usepackage{amsmath}
\usepackage{graphicx}
\usepackage{subfigure}
\usepackage{latexsym}
\usepackage{amsmath}
\usepackage{amsfonts}
\usepackage{amssymb}
\usepackage{tikz}
\usepackage{mathdots}
\usepackage{comment}
\usepackage{float}
\usepackage[mathscr]{euscript}
\usepackage{enumerate}
\usepackage{epsfig}
\usepackage { graphicx }
\usepackage { hyperref }
\usepackage { graphics }
\usepackage { graphicx }

\usepackage {eufrak}
\usepackage[utf8]{inputenc}
\usepackage{longtable}
\usepackage[lite]{amsrefs}
\renewcommand{\PrintDOI}[1]{\href{http://dx.doi.org/\detokenize{#1}}{doi: \detokenize{#1}}%
	\IfEmptyBibField{pages}{, (to appear in print)}{}}

\theoremstyle{definition}
\newtheorem{theorem}{Theorem}[section]

\newtheorem{corollary}[theorem]{Corollary}

\theoremstyle{definition}
\newtheorem{definition}[theorem]{Definition}
\newtheorem{example}[theorem]{Example}
\theoremstyle{remark}
\newtheorem{remark}[theorem]{Remark}
\numberwithin{equation}{section}

\numberwithin{equation}{section}
\title{On Rational Knots and Links in the Solid Torus}

\author{Khaled Bataineh}
\address{Jordan University of Science and Technology, Irbid, Jordan }
\email{khaledb@just.edu.jo}
\author{Mohamed Elhamdadi}
\address{University of South Florida, Tampa, USA }
\email{emohamed@mail.usf.edu}
\author{Mustafa Hajij}
\address{University of South Florida, Tampa, USA}
\email{mhajij@usf.edu}
\date{}
\dedicatory{}
\DeclareMathOperator{\Ima}{Im}

\begin{document}

	\maketitle 
		
\begin{abstract}
We introduce the notion of rational links in the solid torus. We show that rational links in the solid torus are fully characterized by rational tangles, and hence by the continued fraction of the rational tangle. Furthermore, we generalize this by giving an infinite family of ambient isotopy invariants of colored diagrams in the Kauffman bracket skein module of an oriented surface. 
\end{abstract}

\tableofcontents

%\maketitle

\bigskip

\section{Introduction}
 A \emph{$2$-tangle} is a curve $\Gamma$ embedded in a closed $3$-ball decomposed into a union of finite number of circles and two arcs such that $\Gamma$ meets the $3$-ball orthogonally in exactly $4$ points.   
A \emph{diagram} of a $2$-tangle can be thought of as a region in a knot or link diagram surrounded
by a rectangle such that the knot or link diagram crosses the rectangle in
four fixed points. These four points are usually thought of as the four corners NW, NE, SW, SE. See Figure \ref{two_tangles} for an example of a 2-tangle.

\begin{figure}[h]
  \centering
   {\includegraphics[scale=0.15]{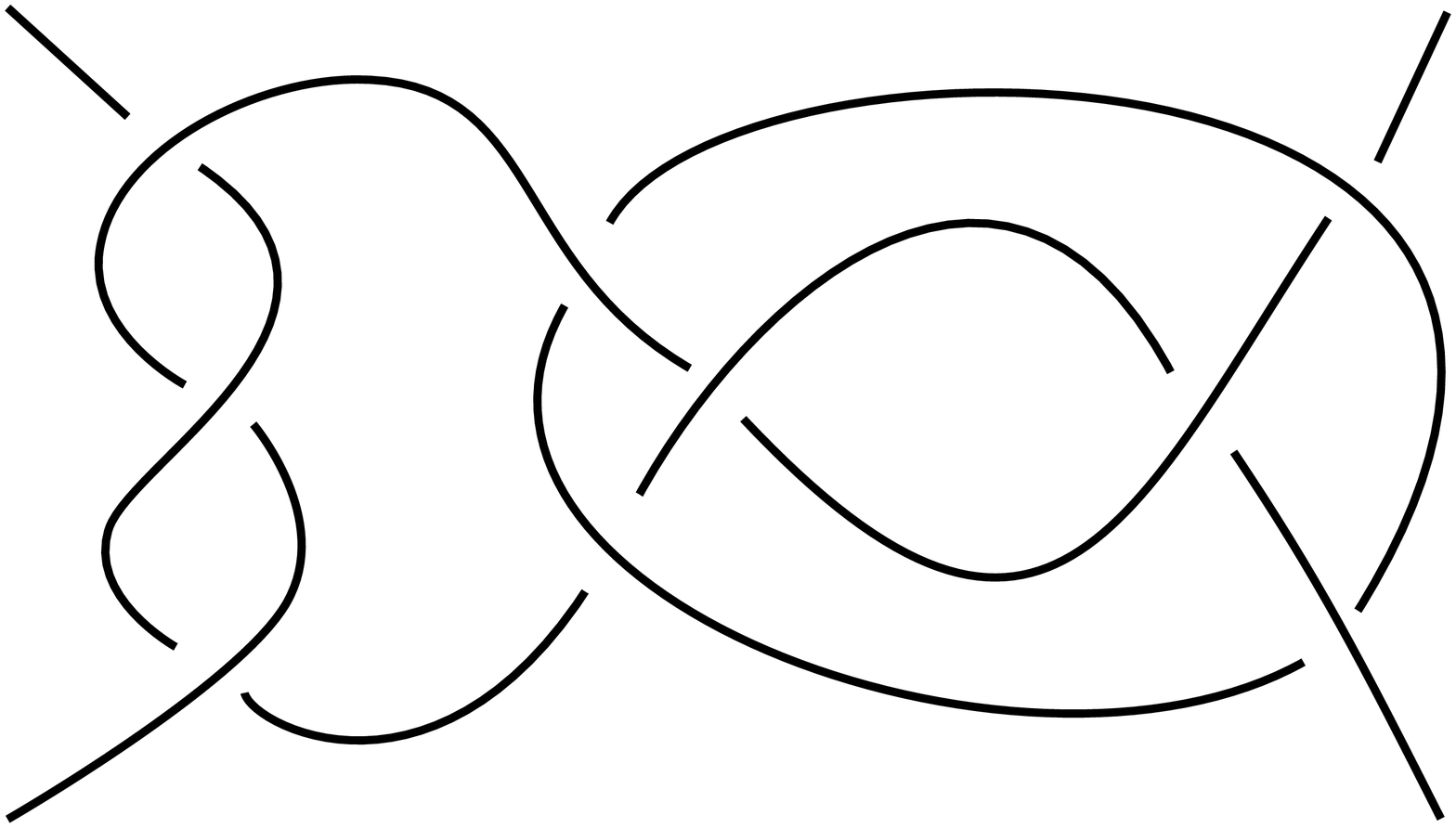}
    \caption{An example of a 2-tangle.}
  \label{two_tangles}}
\end{figure}

A\emph{\ rational tangle} is a special 2-tangle that results from a finite
sequence of twists of consecutive pairs of the four endpoints of two unknotted arcs (See Figure \ref{RT}). Not all 2-tangles are rational. Figure \ref{two_tangles} is an example of a 2-tangle that is not rational. For the formal definition of rational tangles see Section \ref{rational tangles section}.

\begin{figure}[h]
  \centering
   {\includegraphics[scale=0.15]{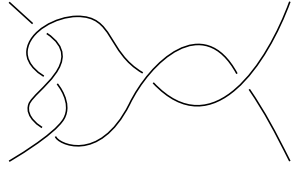}
    \caption{An example of a rational tangle.}
  \label{RT}}
\end{figure}

Rational links are obtained by numerator closures of rational tangles; that is by connecting the two upper endpoints and the two lower endpoints of a given rational
tangle by two unknotted and unlinked arcs. Rational knots and links were considered early in \cite{BS} and \cite{Re}.
The notion of a rational knot was introduced by J. Conway in his work on
classifying knots and links \cite{Co}. Rational tangles and rational knots have been studied extensively in the literature \cite{7mar,BM,BZ,MR1436484,Ka,KL1,Ro}. Furthermore, they have found applications in the study of DNA recombination \cite{KL3,7mar2}.

In this article we define rational knots and links in the solid torus using rational tangles. We show that, with the closure of arcs that we define, two rational links in the solid torus are isotopy equivalent if and only if the corresponding rational tangles are isotopy equivalent. Then we use this result together with the theorem that classifies rational tangles to conclude that the fraction of rational tangles classifies rational knots and links in the solid torus. Finally, we generalize the result by giving an infinite family of ambient isotopy invariants of colored links in the Kauffman bracket skein module of an oriented surface.

{\bf Plan of the Paper:} This paper is organized as follows. In Section \ref{section2} we give the basic concepts and terminology. In Section \ref{rational tangles section} we define rational tangles and their fractions, together with the theorem of classification of rational tangles, and some other basic results. In Section \ref{section3} we introduce the notion of rational links in the solid torus and prove that rational tangles are fully characterized by rational links in the solid torus. In Section \ref{Generalizations} the results are generalized to obtain an infinite family of ambient isotopy invariants of colored diagrams in an oriented surface.

\section{Basic Concepts and Terminology}
\label{section2}

Let $M$ be an oriented $3$-manifold and let $I$ be a closed interval. A framed link in $M$ is  a disjoint union of oriented annuli into $M$. If the manifold $M$ has a boundary $\partial M$ then a closed interval in $\partial M$ is called a framed point. A \text{band} in $M$ is an oriented surface homeomorphic to $I\times I$ embedded in $M$ that  meets $\partial M$ orthogonally at two framed points. 
\begin{definition}\cite{Przytycki, RT} 
	Let $M$ be an oriented $3$-manifold and $\mathcal{R}$ be a commutative ring with an identity and an invertible element $A$. Let $\mathcal{L}_M$ be the set of equivalence classes of framed links in $M$ (including the empty link) up to isotopy. If $M$ has a boundary and an even number of marked framed points on $\partial M$ then the set $\mathcal{L}_M$ also includes the isotopy classes of bands that meet the marked points. Let $\mathcal{R}\mathcal{L}_M$ be the free $ \mathcal{R}$-module generated by the set $\mathcal{L}_M$. The \textit{Kauffman Bracket Skein Module} of $M$ and $\mathcal{R}$, denoted by $\mathcal{S}(M,\mathcal{R})$ is the quotient module
	$\mathcal{S}(M,\mathcal{R})=\mathcal{R}\mathcal{L}_M/K(M)$,
	where $K(M)$ is the smallest submodule of $ \mathcal{R}\mathcal{L}_M$ that is generated by all expressions of the form:
	\begin{eqnarray*}(1)\hspace{3 mm}
		\begin{minipage}[h]{0.06\linewidth}
			\vspace{0pt}
			\scalebox{0.04}{\includegraphics{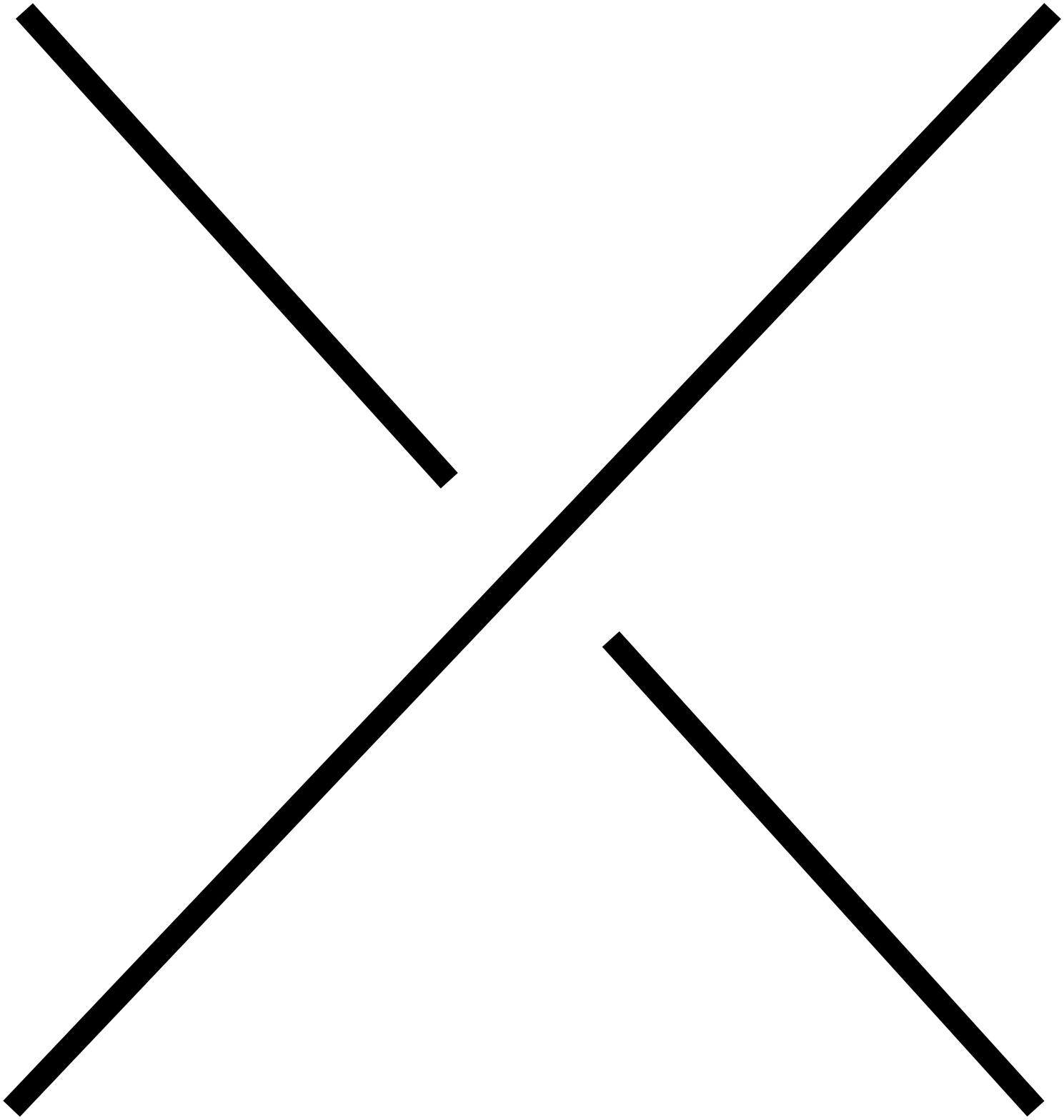}}
		\end{minipage}
		-
		A 
		\begin{minipage}[h]{0.06\linewidth}
			\vspace{0pt}
			\scalebox{0.04}{\includegraphics{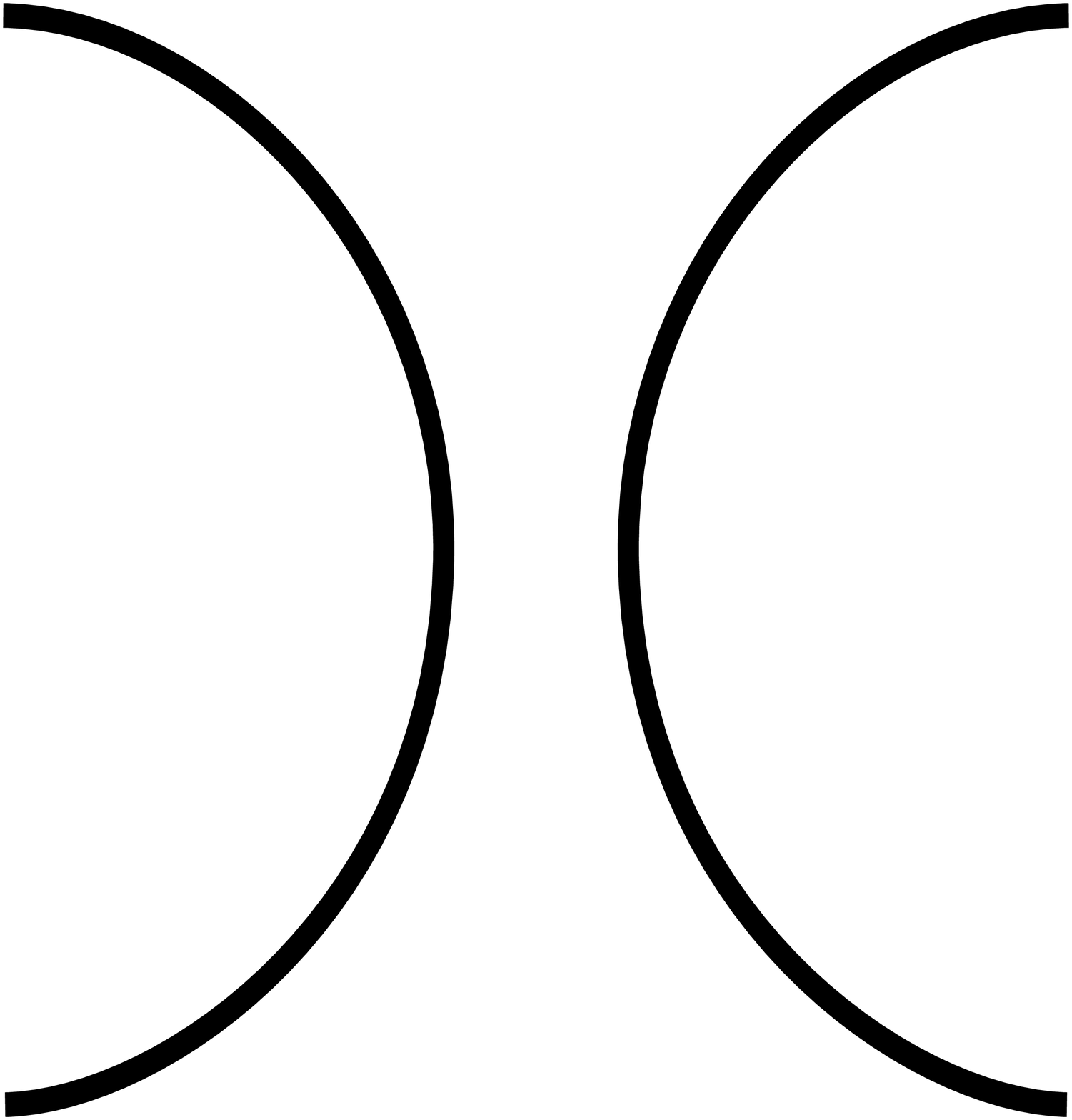}}
		\end{minipage}
		-
		A^{-1} 
		\begin{minipage}[h]{0.06\linewidth}
			\vspace{0pt}
			\scalebox{0.04}{\includegraphics{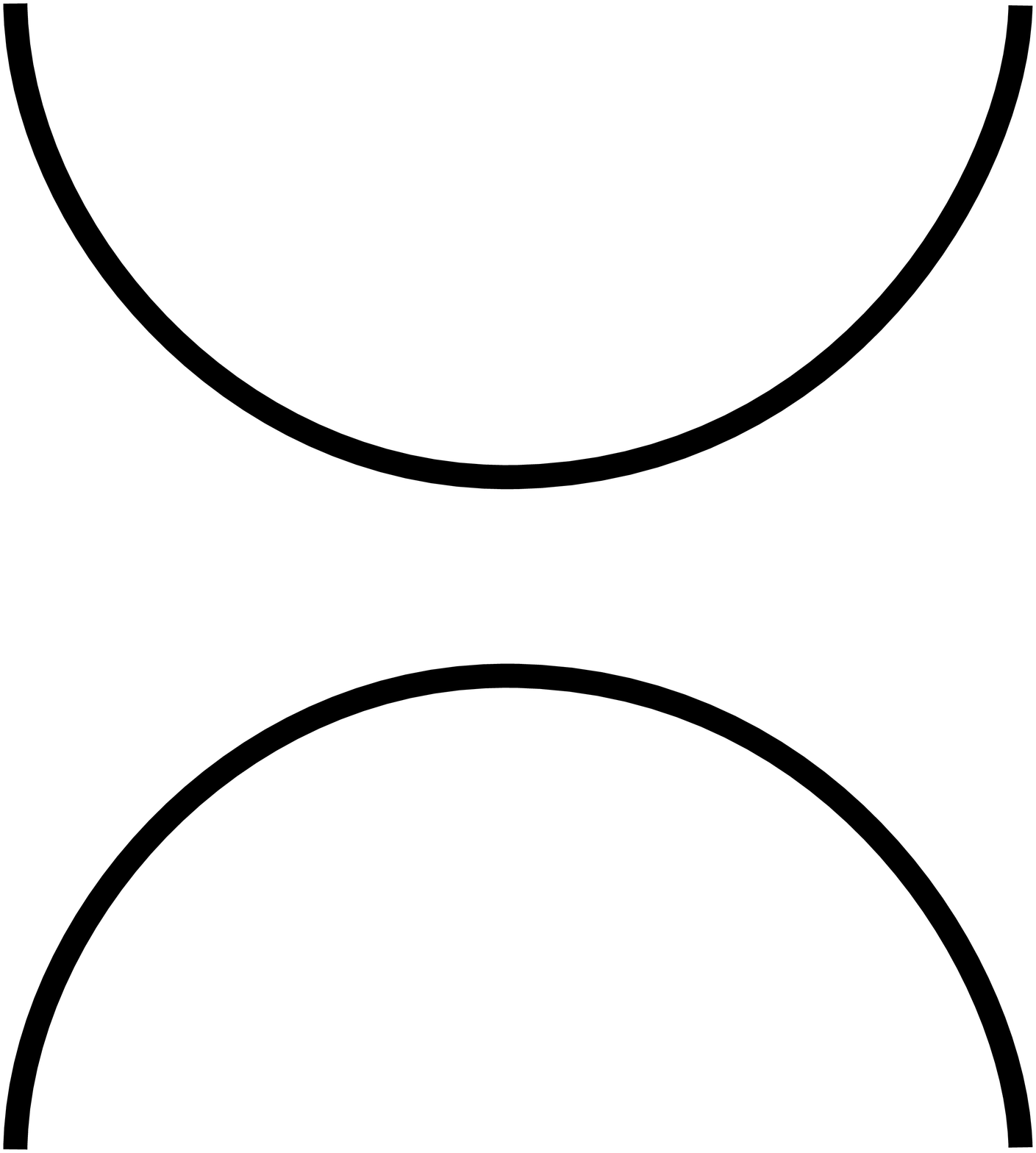}}
		\end{minipage}
		, \hspace{20 mm}
		(2)\hspace{3 mm} L\sqcup
		\begin{minipage}[h]{0.05\linewidth}
			\vspace{0pt}
			\scalebox{0.02}{\includegraphics{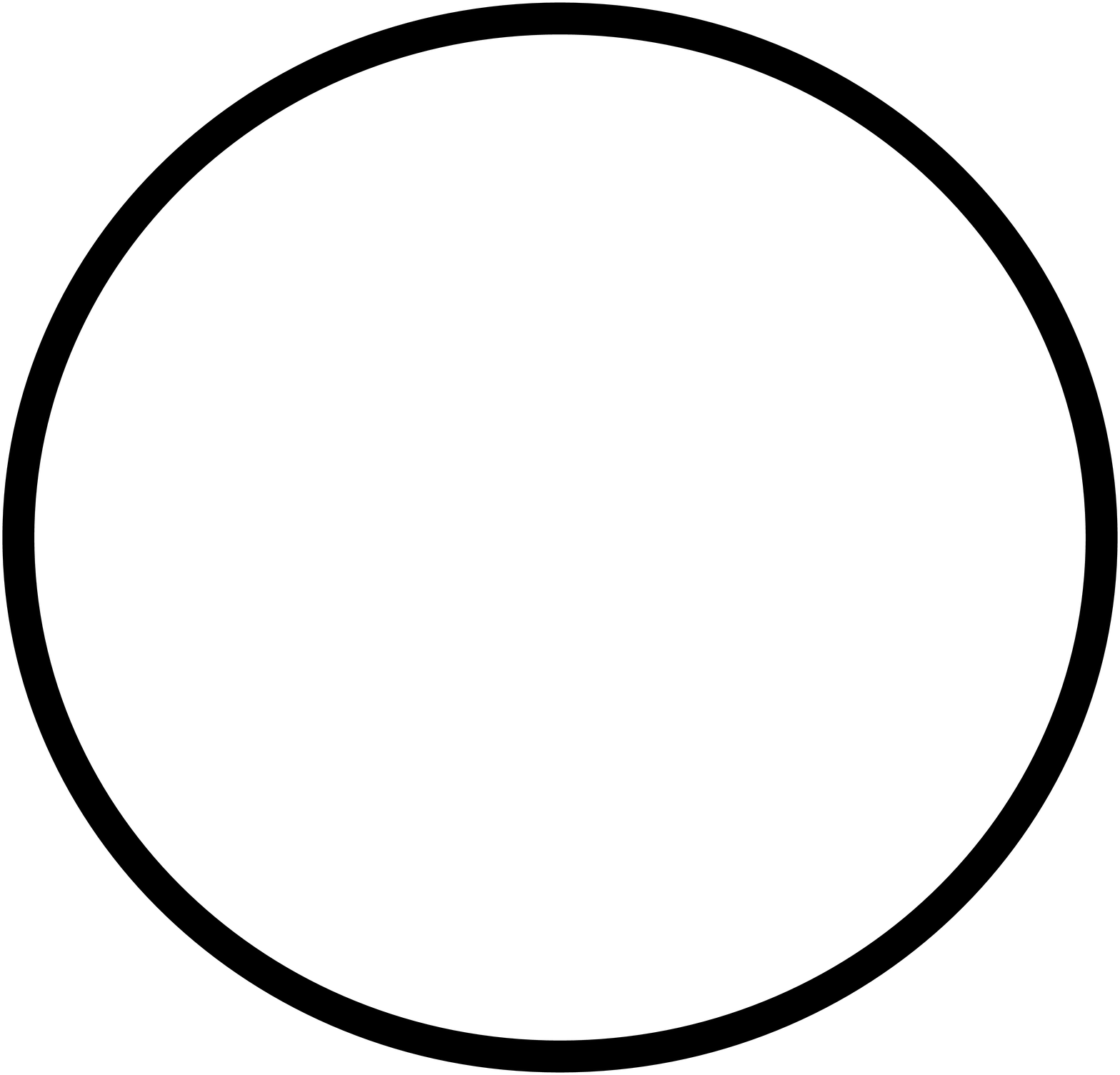}}
		\end{minipage}
		+
		(A^{2}+A^{-2})L. 
	\end{eqnarray*}
	where $L\sqcup$ \begin{minipage}[h]{0.05\linewidth}
		\vspace{0pt}
		\scalebox{0.02}{\includegraphics{simple-circle}}
	\end{minipage}  consists of a framed link $L$ in $M$ and the zero-framed knot 
	\begin{minipage}[h]{0.05\linewidth}
		\vspace{0pt}
		\scalebox{0.02}{\includegraphics{simple-circle}}
	\end{minipage} that bounds a disc in $M$.
\end{definition}    
In this article the ring $\mathcal{R}$ will be the ring of all Laurent polynomials $\mathbb{Z}[A,A^{-1}]$ except for the last section where we will need the field $\mathbb{Q}(A)$ of rational functions in the indeterminate $A$.  Furthermore, if $F$ is an oriented surface and $M=F \times I$ then we will refer to $\mathcal{S}(M)$ by $\mathcal{S}(F)$ and refer to this module by the Kauffman bracket skein module of $F$.

In this paper we will use the skein modules of the sphere $S^{2}$, the disc $D^{2}$ with $2n$ marked points on the boundary, and the annulus $S^1\times I$. The Kauffman bracket skein module of the sphere is isomorphic to the ring $\mathbb{Z}[A,A^{-1}]$. To see this, let $D$ be any diagram in $S^2$. The Kauffman relations can be used to eliminate all crossings and unknots in $D$. Thus we can write $D =\langle D \rangle\phi$ in $\mathcal{S}(S^{2})$, where $\langle D \rangle\in \mathbb{Z}[A,A^{-1}]$ is the normalized Kauffman bracket of $D$. This gives us an isomorphism induced by sending $D$ to $\langle D \rangle$. The relative Kauffman bracket skein module of the disc with $2n$ marked points on the boundary will be denoted by $\mathcal{S}(D^2,2n)$. We fix $n$ points on the top of the disc and $n$ points on the bottom. This module can be made into an associative algebra over $\mathbb{Z}[A,A^{-1}]$  by the obvious diagram vertical juxtaposition known as the \textit{$n^{th}$ Temperley-Lieb algebra} $TL_n$ \cite{jones}. In $TL_n$ there are $c_n$ non-isotopic diagrams consisting of non-crossing arcs connecting the $2n$ boundary points of the disc. These diagrams are usually called \textit{crossingless matching diagrams} on $2n$ nodes. As a $\mathbb{Z}[A,A^{-1}]$ -module, $TL_n$ is spanned by all crossingless matching diagrams on $2n$ nodes. It is known that there are $c_n= \frac{1}{n+1} {2n \choose n} 
$ such diagrams in $TL_n$ where $c_n$ denotes the $n$th Catalan number.\\

 The Kauffman bracket skein module of the annulus is $\mathbb{Z}[A,A^{-1}][z]$. Here $z$ represents the zero-framed knot parallel to $\partial(S^{1}\times I)$ and $z^n$ represents $n$ parallel copies of $z$. See Figure \ref{basis}

\begin{figure}[h]
  \centering
   {\includegraphics[scale=0.2]{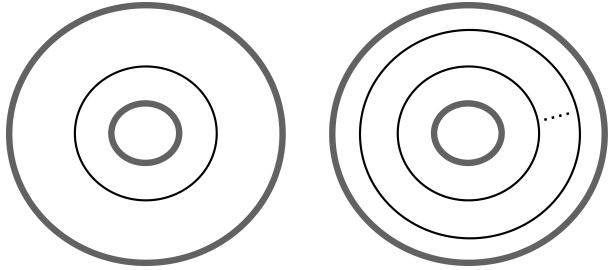}
    \caption{On the left $z$ and on the right $z^n$.}
  \label{basis}}
\end{figure}
\subsection{Wiring Maps}

Let $F$ and $F^{\prime}$ be two oriented surfaces with marked points. 
%Suppose that the surface $F$ is embedded in $F^{\prime}$ such that the marked points in $F$ coincide with those of $F^{\prime}$. 
A \textit{wiring map} \cite{mortan} from $F$ to $F^{\prime}$ is an embedding of $F$ into $F^{\prime}$ such that the marked points in $F$ coincide with those of $F^{\prime}$ and a fixed \textit{wiring diagram} of arcs and curves in $F^{\prime}-F$ such that the marked points of the arcs consist of the marked points in $F$ and $F^{\prime}$. In this way a diagram in $F$ induces a diagram $\mathcal{W}(D)$ in $F^{\prime}$ by extending $D$ by the wiring diagram. It is clear that a wiring $\mathcal{W}$ from $F$ to $F^{\prime}$ induces a module homomorphism:
\begin{equation*}
\mathcal{S}(\mathcal{W}):\mathcal{S}(F)\longrightarrow \mathcal{S}(F^{\prime}).
\end{equation*} 

\begin{example}
Consider the square $I \times I$ with $n$ marked points on the top edge
and $n$ marked points on the bottom edge. Embed $I\times I$ in $S^1\times I$ and join the $n$ points
on the top edge to the $n$ points on the bottom edge by parallel arcs as follow:
\begin{figure}[H]
  \centering
   {\includegraphics[scale=0.45]{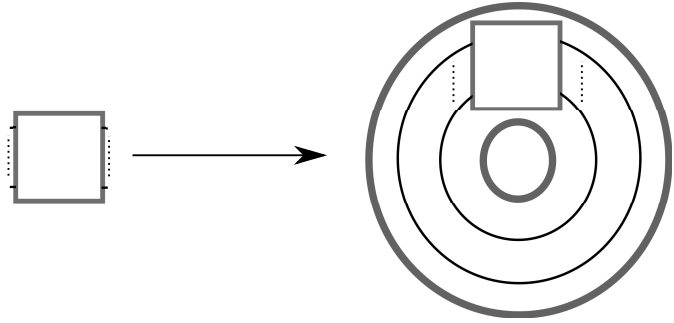}
    \caption{Mapping a disc with $2n$ marked points on the boundary to the annulus torus.}
  \label{wire}}
\end{figure}
For each $n$ this wiring defines a map :
\begin{equation}
\label{src}
\mathcal{STC}_n : TL_n \longrightarrow \mathcal{S}(S^1 \times I).
\end{equation}
For a diagram $D$ in $TL_n$ we call $\mathcal{STC}_n$ the solid torus closure of $D$.
\subsection{The Jones Wenzl Projector}
In section \ref{Generalizations} we extend our work on rational links in the solid torus to colored diagrams in oriented surfaces. For this reason we will also need the definition of the Jones-Wenzl projector \cite{jones2}. Here we utilize the graphical definition of the Jones-Wenzl projector $f^{(n)}$ as given by \cite{Lickbook}. The $n^{th}$ Jones-Wenzl projector is a special idempotent element in $TL_n$ that can be characterized by the following graphical equations: 

\begin{eqnarray}
\label{properties}
\hspace{0 mm}
    \begin{minipage}[h]{0.21\linewidth}
        \vspace{0pt}
        \scalebox{0.115}{\includegraphics{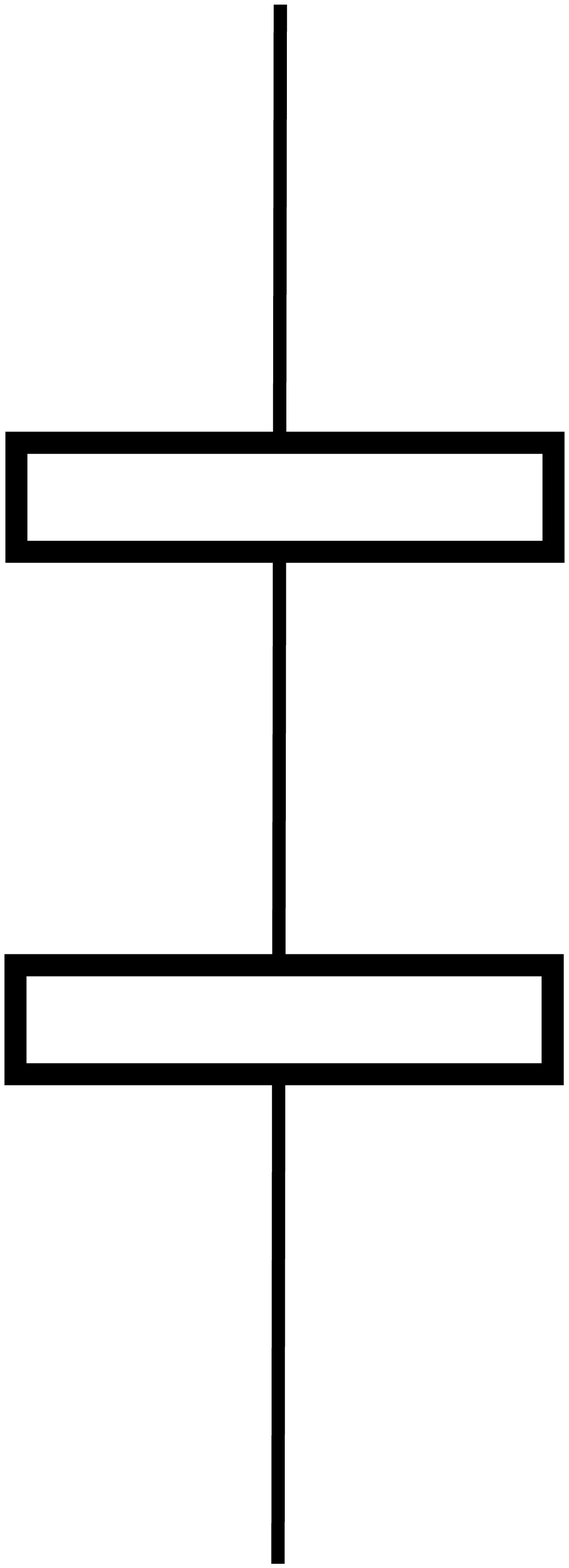}}
        \put(0,+80){\footnotesize{$n$}}
       
   \end{minipage}
  = \hspace{5pt}
     \begin{minipage}[h]{0.1\linewidth}
        \vspace{0pt}
         \hspace{50pt}
        \scalebox{0.115}{\includegraphics{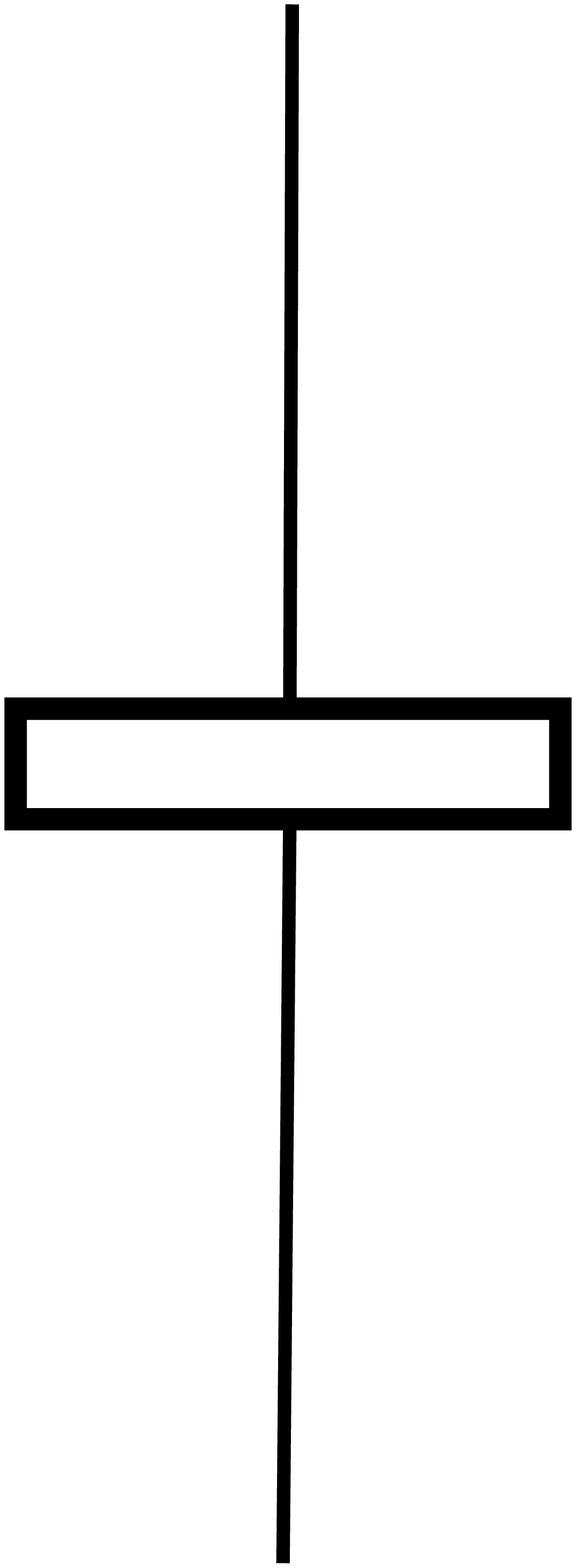}}
        \put(-60,80){\footnotesize{$n$}}
   \end{minipage}
    , \hspace{15 mm}
    \begin{minipage}[h]{0.09\linewidth}
        \vspace{0pt}
        \scalebox{0.115}{\includegraphics{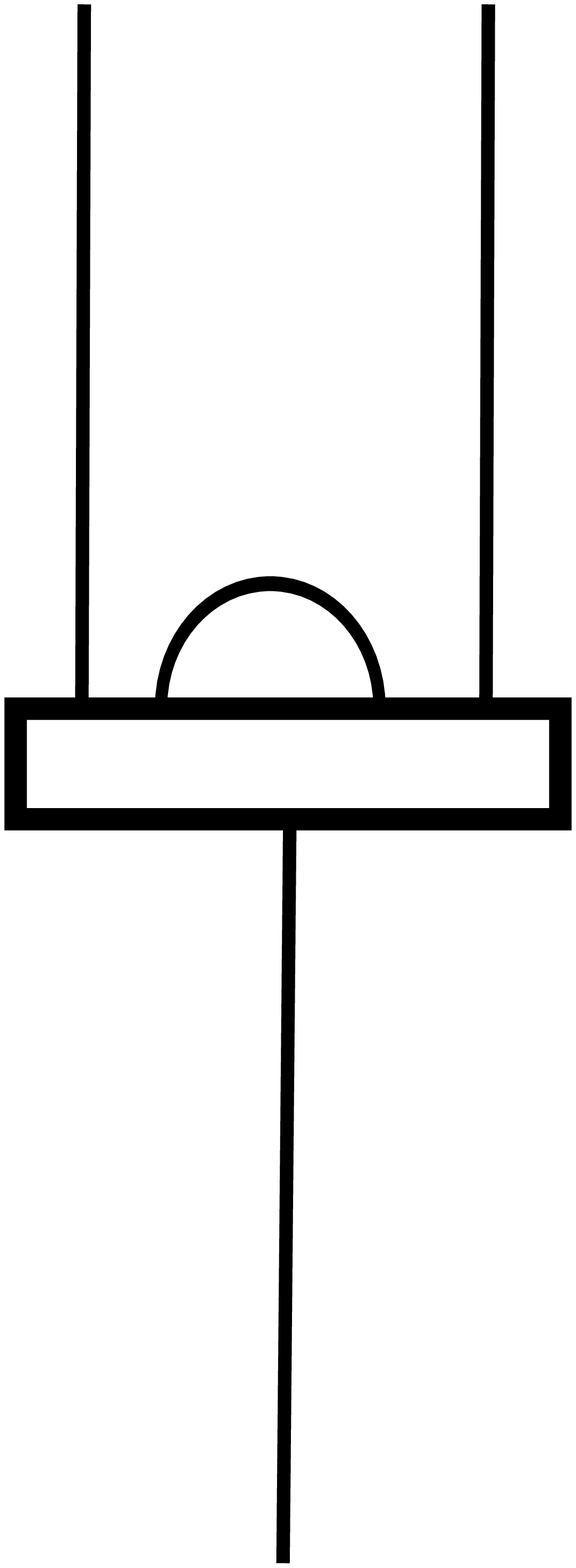}}
         \put(-70,+82){\footnotesize{$n-i-2$}}
         \put(-20,+64){\footnotesize{$1$}}
        \put(-2,+82){\footnotesize{$i$}}
        \put(-28,20){\footnotesize{$n$}}
   \end{minipage}
   =0,
   \label{AX}
  \end{eqnarray}
Here the idempotent $f^{(n)}$ is depicted by the box with $n$ strands coming from the upper part of the box and $n$ strands leaving the lower part.  For other applications of the Jones-Wenzl projector see \cite{armond,EH,hajij4,hajij2,hajij3}. We will also need the following equation:    
  \begin{eqnarray}
  \label{thisthis}
     \begin{minipage}[h]{0.08\linewidth}
        \vspace{-5pt}
        \scalebox{0.19}{\includegraphics{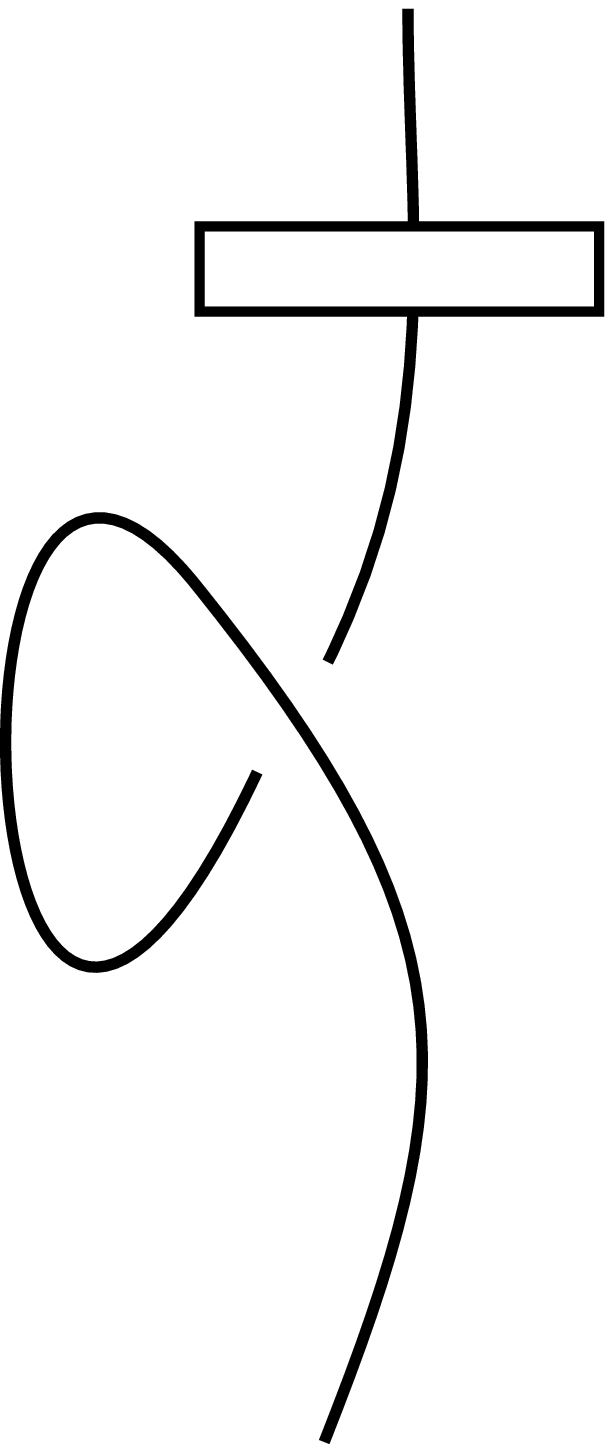}}
        \put(-22,+75){\footnotesize{$n$}}
   \end{minipage}
  =\mu_n
     \begin{minipage}[h]{0.09\linewidth}
        \hspace{5pt}
        \scalebox{0.19}{\includegraphics{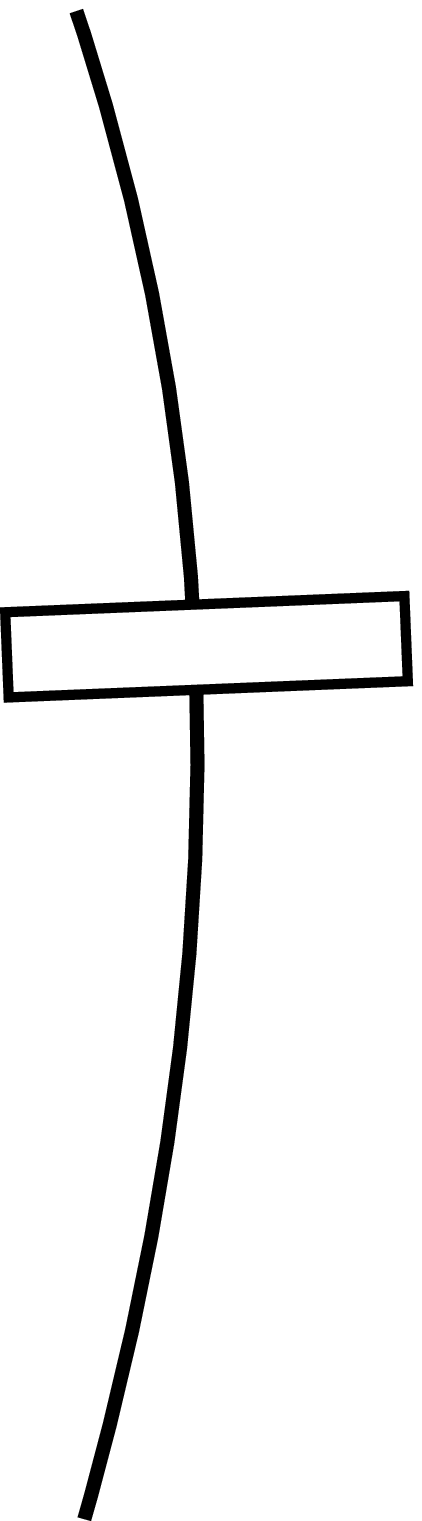}}
        \put(-6,55){\footnotesize{$n$}}
   \end{minipage},
   \end{eqnarray}

where $\mu_n=(-1)^nA^{-n^2-2n}$. For more details about the Jones-Wenzl projector and its properties see \cite{Lic92}.
\subsection{Colored Kauffman Bracket Skein Modules} Some of our results can be generalized to colored links on in the Kaufman bracket skein module of oriented surfaces with marked points on the boundary. For this purpose we introduce these modules here.

Consider the skein module of the disc $D$
with $a_{1}+...+a_{m}$ marked points on the boundary. In this article we are interested in certain submodules of $D$. This submodule is constructed as follows. Partition the set
of $n=a_{1}+...+a_{m}$ marked points on the boundary of $D$ into $m$ clusters such that the $i^{th}$ cluster has $a_{i}$ points for $1 \leq i \leq n$. We place a Jones-Wenzl idempotent $f^{(a_{i})}$ at each cluster of $a_{i}$ points. Note that an element of this skein module is obtained by taking an element in the module of the disc with $a_{1}+...+a_{m}$ marked points and then adding the idempotents  $f^{(a_1)},...,f^{(a_m)}$ on the outside of the disc. In particular when we have $2m$ clusters of points on the boundary of $D$ each one of them with exactly $n$ points ($m$ clusters on the top of $D$ and $m$ on the bottom) then this space forms a generalization of the Temperley-Lieb algebra called \textit{the colored Temperley-Lieb algebra} and is denoted by $TL^{n}_{m}$. Note that $TL^{1}_{m}=TL_{m}$.  See Figure \ref{disc example} for an example of such a module. Finally, if $D$ has $k$ clusters of $n$ points on the top and $l$ clusters of $n$ points on the bottom (and again place the idempotents on the outside at the clusters) then we will denote this space by $\mathcal{D}_{l}^{k}(n)$. This construction can be generalized to any $3$-manifold $M=F\times I$ where $F$ is an orientable surface with a boundary. In other words, one can consider the Kauffman bracket skein module with $a_1,...,a_m$ marked points on the boundary of $F$ and stick idempotents on the outside of $F$. This module will be denoted by $\mathcal{S}(F_{a_1,...,a_m})$. When all the cardinalities of the clusters are all the same, say $n$, we will denote this module by $\mathcal{S}(F_{m}(n))$. Figure \ref{disc example} gives an example of such construction where $F$ is the disk $D$ and we have 4 cluster of boundary points $a_1$, $a_2$, $a_3$ and $a_4$ the cardinality of these clusters is $n$, i.e., $a_i=n$ for all $i$. 
\begin{figure}[H]
  \centering
   {\includegraphics[scale=0.16]{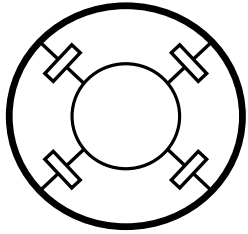}
    \caption{The Kauffman bracket skein module $TL^n_4$.}
  \label{disc example}}
\end{figure}

\end{example}

\subsection{Main Results}
In this section we state the main result of this article.

We define rational knots and links in the solid torus and we prove that they are completely characterized with rational tangles:\\

{ \bf Theorem} 
{ \ref{main theorem} \em  
Let $T_{1}$ and $T_{2}$ be two rational tangles, then $\overline{T_{1}}\simeq \overline{T_{2}}$ if and only if $\
T_{1}\simeq T_{2}$.
}\\

Here $\overline{T}$ is the closure of a $2$-tangle $T$ inside the solid torus defined in Figure \ref{closures1}.  
Theorem \ref{main theorem} allows us to bring all the tools available in the study of rational tangles to use them in the study of rational knots in the solid torus. The main ingredient of this result is a theorem by Goldman and Kauffman \cite{MR1436484}, which states that the ratio of the coefficients of the expansion of a 2-tangle $T$ in standard Kauffman basis is an ambient isotopy invariant of the tangle $T$. We generalize this result to colored diagrams in any oriented surface:\\

{ \bf Theorem} 
{ \ref{main theorem2} \em  Let $D$ be a diagram in a surface $F$ with $m$ marked points. Let $D^{n}$ be the diagram obtained from $D$ by coloring all its components with the $n^{th}$ Jones-Wenzl projector. Let $\mathcal{E}= \{E_i\}_{i=0}^s$ be a basis for the skein module $\mathcal{S}(F_{m}(n))$ and let $\gamma_i(D)$ be the $i^{th}$ coefficient obtained from expanding $D^n$ in terms of the basis $\mathcal{E}$. Then for any $ 0\leq i \leq s-1 $ the ratios:
\begin{equation*}
CR^i_D(A)=\frac{\gamma_i(D)}{\gamma_{s}(D)} 
\end{equation*} 
are ambient isotopy invariants of the  diagram $D$.
}\\

\section{Rational Tangles} \label{rational tangles section}
The $2$-tangles in Figure \ref{smoothings} are denoted from left to right by $\left[ \infty \right] $,   $\left[ 0\right] $, $%
\left[ 1\right] $ and $\left[ -1\right] $.

\begin{figure}[H]
  \centering
   {\includegraphics[scale=0.2]{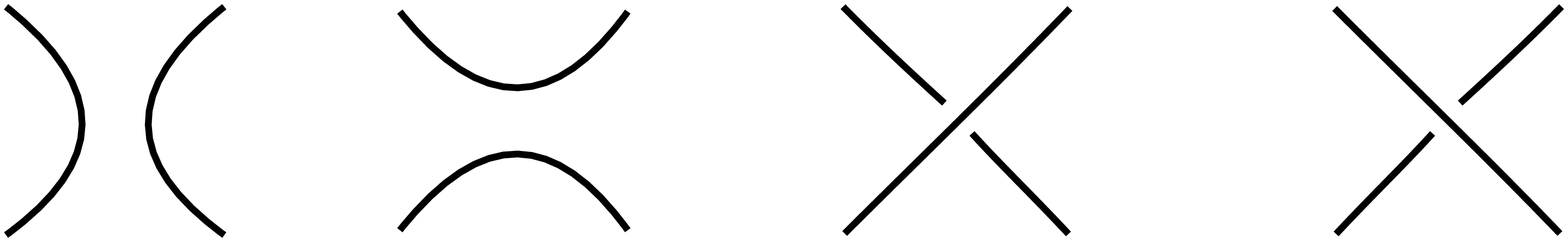}
   \vspace{10pt}
    \caption{From left to right : the $\left[ \infty \right] $ tangle, the $\left[ 0\right] $ tangle, the $%
\left[ 1\right] $ tangle and the $\left[ -1\right] $ .}
  \label{smoothings}}
\end{figure}
\vspace{-20pt}
 \emph{An integer tangle} $\left[ n\right] $ is the result of
winding two horizontal strands around each other to get a tangle that
involves $\left\vert n\right\vert $ crossings. If the upper strand at each crossing has positive slope, then $n$ is positive. Otherwise, $n$ is negative. See Figure \ref{three} for the tangle $\left[ 3\right] $ and the
tangle $\left[ -3\right] $.  Note that some authors use the opposite convention such as in \cite{dna}.
\begin{figure}[h]
  \centering
   {\includegraphics[scale=0.2]{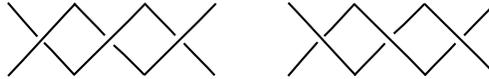}
    \caption{Examples of integer tangles. On the left the tangle $[3]$ and on the right the tangle $[-3]$.}
  \label{three}}
\end{figure}

Let $n$ be an integer representing a tangle diagram $\left[ n\right] .$
The mirror image of $\left[ n\right] $ in the diagonal line
connecting NW and SE is denoted by the tangle diagram $\dfrac{1}{%
\left[ n\right] }$ (see Figure \ref{this isstrange} for the tangles $\left[ 4\right] $ and $%
\dfrac{1}{\left[ 4\right] }$, respectively).
\begin{figure}[h]
  \centering
   \includegraphics[scale=0.2]{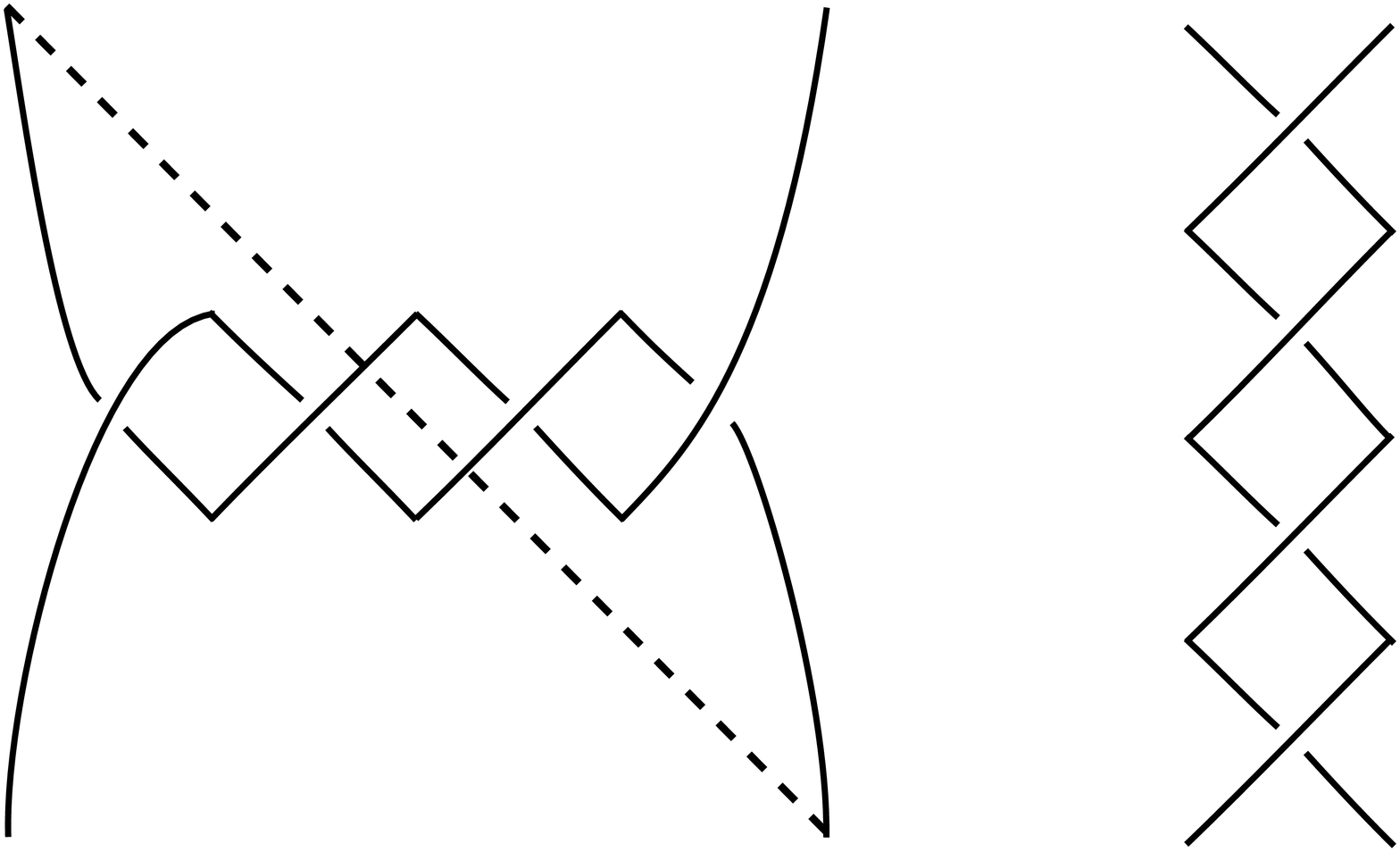}
    \caption{The tangle $[4]$ on the left and the tangle $\frac{1}{[4]}$ on the right}
    \label{this isstrange}
\end{figure}

Let $a_{1}$, $a_{2}$, $\cdots $, $a_{m}$ be a finite sequence of integers. The \textit{rational tangle} $\left[ a_{1}\text{ }a_{2}\cdots a_{m}\right] $ is constructed as follows. We distinguish two cases depending on the parity of $m$. If $m$ is an odd number, draw the horizontal
diagram $\left[ a_{1}\right] $ on the top and draw $\dfrac{1}{\left[ a_{2}%
\right] }$ on the bottom and connect the close ends together then draw the
diagram $\left[ a_{3}\right] $ to the right and connect the close ends
together then draw the diagram $\dfrac{1}{\left[ a_{4}\right] }$ below and
connect the close ends together. Continuing in this manner gives a unique
tangle diagram denoted by $\left[ a_{1}\text{ }a_{2}\cdots a_{m}\right] $. If $m$ is even, we start with the vertical tangle diagram $\dfrac{1}{\left[
a_{1}\right] }$ on the left and the horizontal tangle diagram $\left[ a_{2}%
\right] $ on the right and we connect the close ends together then draw the
diagram $\dfrac{1}{\left[ a_{3}\right] }$ below the previous tangle and
connect the close ends together. Continue this process alternatively to get
the diagram of the tangle $\left[ a_{1}\text{ }a_{2}\cdots a_{m}\right] .$ 
See Figure \ref{examples of rational tangles} for some examples. More details can be found in \cite{MR2079925}. In the next section we will also give a definition of the rational tangles that utilizes tangle operations.
\begin{figure}[h]
  \centering
   {\includegraphics[scale=0.5]{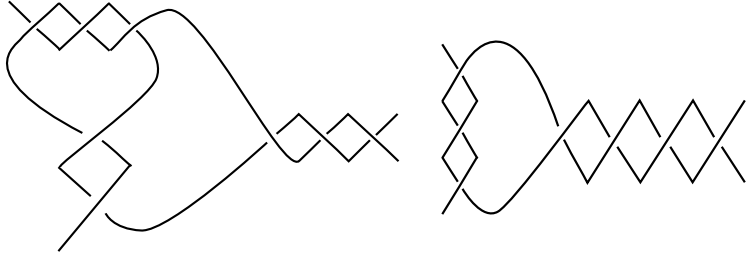}
    \caption{The tangle $[3,2,-3]$ on the left and the tangle $[3,4]$ on the right}
  \label{examples of rational tangles}}
\end{figure}
\subsection{Operations on Tangles} Let $T$ and $S$ be  $2$-tangles. The tangles $T+S,$ $-T,$ and $1/T$ are defined as illustrated in Figure \ref{operations}, \begin{figure}[H]
  \centering
   {\includegraphics[scale=0.15]{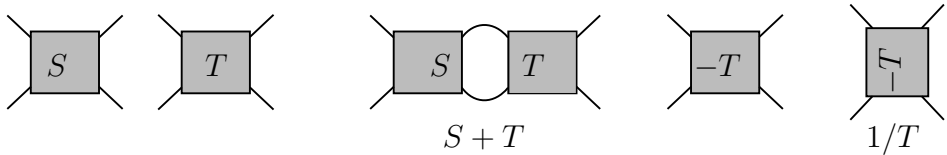}
    \put (-195,17) {\large{$S$}}
    \put (-160,17) {\large{$T$}}
    \put (-190,-10) {\large{$S+T$}}
    \put (-340,17) {\large{$S$}}
    \put (-280,17) {\large{$T$}}
    \put (-95,17) {\large{$-T$}}
    \put (-25,10) {\large{\rotatebox{90}{$-T$}}}
    \put (-30,-10) {\large{$1/T$}}
    \caption{Operations on tangles. }
    \label{operations}
  }
\end{figure}

\noindent Here $-T$ is the mirror image of $T$ and $1/T$ is the mirror image of $T$  through the diagonal line connecting NW and SE. Now we give the definition of rational tangles and their fractions.

Let $a_1,...,a_m$ be integers such that $a_{2},\ldots ,a_{m-1}\in 
%TCIMACRO{\U{2124} }%
%BeginExpansion
\mathbb{Z}
%EndExpansion
-\left\{ 0\right\} $, $a_{1}\in 
%TCIMACRO{\U{2124} }%
%BeginExpansion
\mathbb{Z}
%EndExpansion
$ and $a_{m}\in 
%TCIMACRO{\U{2124} }%
%BeginExpansion
\mathbb{Z}
%EndExpansion
\cup \left\{ \infty \right\}.$ The rational tangle $\left[ a_{1}\text{ }a_{2}\cdots a_{m}\right] $ is constructed from the tangles $[a_i]$, $1 \leq  i \leq m$, via the tangle operations as follows:
\begin{equation*}
\label{fraction}
\left[ a_{1}\text{ }a_{2}\cdots a_{m}\right] :=\left[ a_{m}\right] +\dfrac{1%
}{\left[ a_{m-1}\right] +\dfrac{1}{%
\begin{array}{cc}
\ddots &  \\ 
& \dfrac{1}{\left[ a_{2}\right] +\dfrac{1}{\left[ a_{1}\right] }}%
\end{array}%
}}.
\end{equation*}%

 Note that the following identity is true%
\begin{equation*}
\left[ \infty \right] =\dfrac{1}{\left[ 0\right] }.
\end{equation*}%

%The right hand side of equation \ref{fraction} is called the fraction of the tangle $[a_1,a_2,\cdot,a_m]$.
Note that we use the convintion that the order in the notation of a rational tangle is the order of creation of the tangle, which agrees with the notation used in \cite{MR2079925}. In fact many other articles use a reversed ordering, which rather agrees with the order in the continued fraction, where the continued fraction is defined as follows.
\bigskip Let $T=\left[ a_{1}\ldots a_{m}\right] $ be a rational tangle, with 
$a_{1},\ldots ,a_{m-1}\in 
%TCIMACRO{\U{2124} }%
%BeginExpansion
\mathbb{Z}
%EndExpansion
-\left\{ 0\right\} ,$ and $a_{m}\in 
%TCIMACRO{\U{2124} }%
%BeginExpansion
\mathbb{Z}
%EndExpansion
.$ The \emph{continued fraction}, or simply the fraction, $F(T)$ is
\begin{equation*}
F(T)=a_{m}+\dfrac{1}{a_{m-1}+\dfrac{1}{%
\begin{array}{cc}
\ddots &  \\ 
& \dfrac{1}{a_{2}+\dfrac{1}{a_{1}}}%
\end{array}%
}}.
\end{equation*}
 The following classification theorem is due to Conway \cite{Co}.

\begin{theorem}
Two rational tangles are isotopy equivalent if and only if they have the
same fraction.
\end{theorem}
Let $\left\langle T\right\rangle =\alpha (A)\left\langle \left[ \infty \right]
\right\rangle +\beta (A)\left\langle \left[ 0\right] \right\rangle$ be the bracket of a 2-tangle $T$, where $\alpha(A)$ and $\beta(A)$ in $%
%TCIMACRO{\U{2124} }%
%BeginExpansion
\mathbb{Z}
%EndExpansion
\lbrack A,A^{-1}]$. Since $[0]$ and $[\infty]$ form a basis for $TL_2$, then $\alpha(A)$ and $\beta(A)$ are regular isotopy invariants of the tangle $T$. Furthermore, their ratio is an ambient isotopy invariant as can be seen from the following theorem of Goldman and Kauffman \cite{MR1436484}.

\begin{theorem}
\label{Kauffman Goldman}
Let $T$ be a rational tangle and let $\left\langle T\right\rangle =\alpha
\left( A\right) \left\langle \infty \right\rangle +\beta \left( A\right)
\left\langle 0\right\rangle $ be its bracket, then for any choice of $A,$%
\begin{equation*}
R_{T}(A)=\frac{\alpha \left( A\right) }{\beta \left( A\right) }
\end{equation*}%
is an ambient isotopy invariant of tangles.
\end{theorem}

In particular when $A=\sqrt{i}$, where $i$ is the imaginary complex number, we define
\begin{equation*}
C(T):= -iR_{T}(\sqrt{i})=-i\frac{\alpha \left( \sqrt{i}\right) }{\beta
\left( \sqrt{i}\right) }.
\end{equation*}%
The following theorem was proved in \cite{MR1436484}. 
\begin{theorem}
Let $T$ be a rational tangle, then%
\begin{equation*}
C(T)=F(T).
\end{equation*}
\end{theorem}

\section{Rational Knots and Links in the Solid Torus}
\label{section3}

In this section we introduce the notion of rational knots in the solid torus. Furthermore, we show that rational knots in the solid torus completely characterize rational tangles. We apply this result to completely classify the isotopy type of rational knots and links in the solid torus. We start with a few definitions.
\begin{definition}
 The \emph{solid torus closure} of a 2-tangle $T$ is the closure of $T$ defined in Figure \ref{closures1}, where the projection disc of $T$ is embedded in the annulus. In the figure, the dot represents the complementary solid torus in the $3$-sphere.
 \end{definition}
 
\begin{figure}[h]
  \centering
   {\includegraphics[scale=0.1]{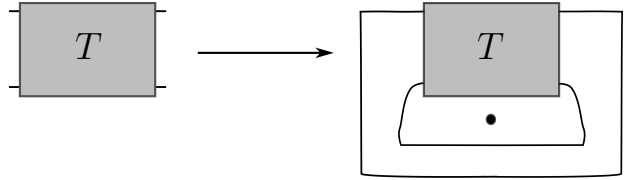}
    %\vspace{-50pt}
     \put (-55,45) {\Large{$T$}}
     \put (-207,45) {\Large{$T$}}
    \caption{On the left : the 2-tangle $T$. On the right : the solid torus closure $\overline{T}$ of the 2-tangle $T$.}
  \label{closures1}}
\end{figure}

\begin{definition}
Let $L$ be a link in the solid torus. We say that $L$ is a \emph{rational link in the solid torus} if there exists a rational tangle $T$ such that $\overline{T}=L$.
\end{definition}
For example, Figure \ref{closure111} shows the closure of the rational tangle
$\left[ 3\text{ }2\text{ }-3\right] $.

\begin{figure}[H]
  \centering
   {\includegraphics[scale=0.1]{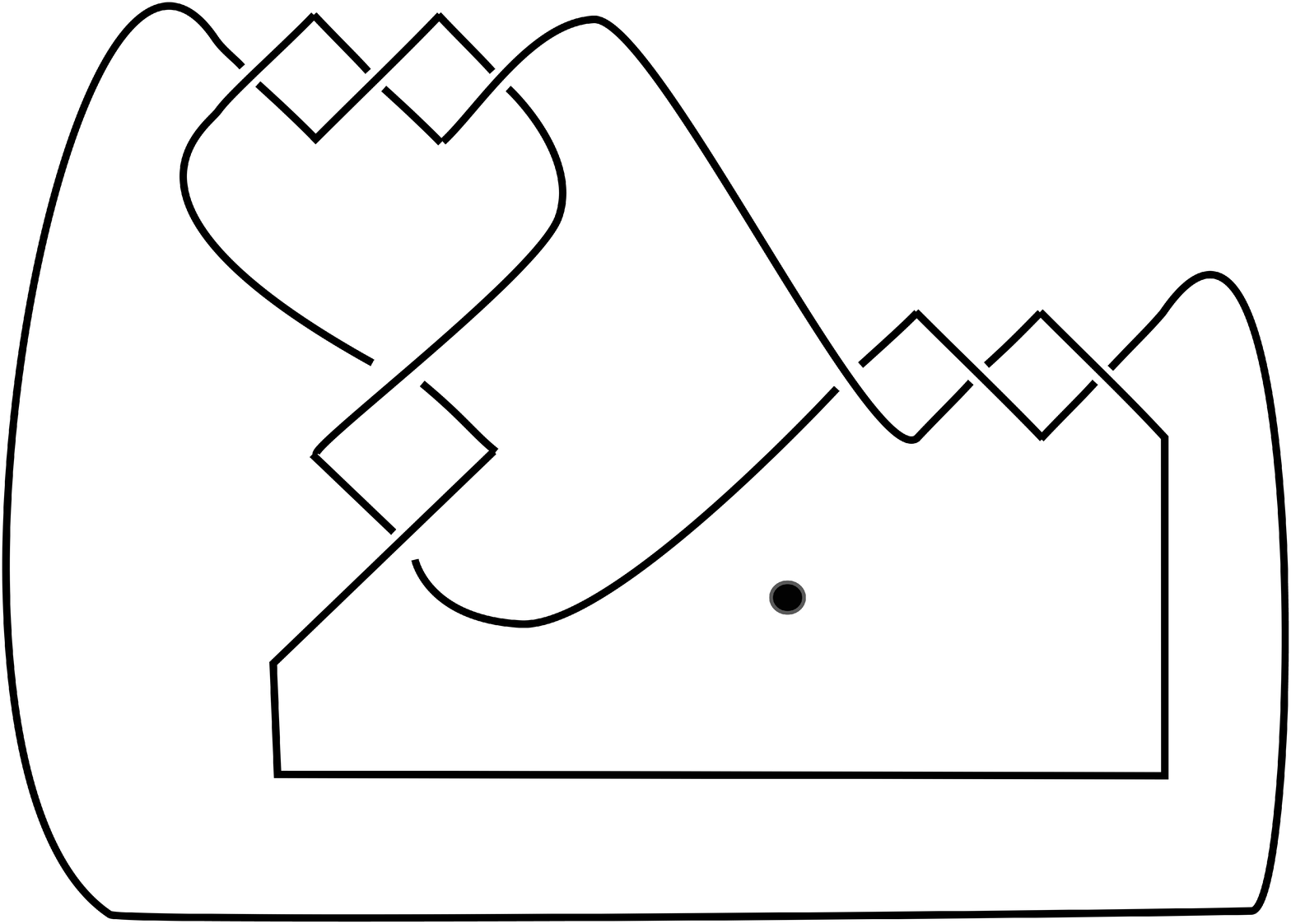}
    \caption{The closure of the rational tangle $\left[ 3\text{ }2\text{ }-3\right] $. }
  \label{closure111}}
\end{figure}

 The following theorem establishes the relationship between the isotopy of a rational link in the solid torus and the isotopy of the rational tangle.
\begin{theorem}
\label{main theorem}
Let $T_{1}$ and $T_{2}$ be two rational tangles, then $\overline{T_{1}}\simeq \overline{T_{2}}$ if and only if $\
T_{1}\simeq T_{2}$.
\end{theorem}

\begin{proof}
Let $T_{1}\simeq T_{2}$.\ Note that the sequence of Reidemeister
moves that takes $T_{1}$ into $T_{2}$ takes place inside the rectangles
defining the tangles $T_{1}$ and $T_{2}$ and hence inside the corresponding
knot (link) diagrams. Therefore, by construction of $\overline{T_{1}}$ and $%
\overline{T_{2}}$, the same sequence of moves takes $\overline{T_{1}}$ into $%
\overline{T_{2}}.$ Hence $\overline{T_{1}}\simeq \overline{T_{2}}.$

In order to prove the converse we first show the relation between rational knots in the solid torus and rational tangles using the wiring map \ref{src}.  Let $\overline{T}$ be a rational
knot in the solid torus such that%
\begin{equation*}
\left\langle T\right\rangle =\alpha (A)[ \infty] +\beta (A)[ 0]  ,
\end{equation*}
On the other hand, using the map \ref{src},

\begin{eqnarray*}
\mathcal{STC}_2(T)&=&\left\langle \overline{T}\right\rangle = \alpha (A) \mathcal{STC}_2([\infty]) +\beta (A) \mathcal{STC}_2( [0] )\\&=& \alpha (A) \delta  +\beta (A) z^2.
\end{eqnarray*}

Define
\begin{equation*}
R_{\overline{T}}(A)=\frac{\alpha \left( A\right) }{\beta \left( A\right) },%
\text{ and }C(\overline{T})=-i\frac{\alpha \left( \sqrt{i}\right) }{\beta
\left( \sqrt{i}\right) }.
\end{equation*}

Note that $\alpha (A)$ and $\beta (A)$ are regular isotopy invariants of
knots in the solid torus. Moreover, as the bracket changes under a type-I
move by multiplying by either $-A^{3}$ or $-A^{-3},$ it is obvious that $%
\alpha (A)$ and $\beta (A)$ both are multiplied by the same factor, and
hence the ratio $R_{\overline{T}}(A)$ does not change under the first
Reidemeister move. Therefore $R_{\overline{T}}(A)$ and $C(\overline{T})$ are
ambient isotopy invariants for rational knots in the solid torus.

By definition of $C(\overline{T})$, we have%
\begin{equation*}
C(\overline{T})=C(T).
\end{equation*}

Since $C(T)=F(T)$, we have%
\begin{equation*}
C(\overline{T})=F(T).
\end{equation*}

Now let $\overline{T_{1}}\simeq \overline{T_{2}}.$ Since $C(\overline{T})$
is an ambient isotopy invariant for rational knots in the solid torus, we
have 
\begin{equation*}
C(\overline{T_{1}})=C(\overline{T_{2}}),
\end{equation*}%
therefore%
\begin{equation*}
F(T_{1})=F(T_{2}),
\end{equation*}%
and hence%
\begin{equation*}
T_{1}\simeq T_{2}.
\end{equation*}
\end{proof}
 The previous theorem has some immediate consequences. We list some of them next. Now we can give the following definition.
\begin{definition}
Let $K$ be a rational knot in the solid torus. Let $T$ be its corresponding tangle.  The fraction of the knot $K$ is defined to be $
F(K):=F(T).
$
\end{definition}
The following Corollary is immediate.

\begin{corollary}
Two rational knots in the solid torus are isotopy equivalent if and only if
they have the same  fraction.
\end{corollary}
In \cite{KL1} the authors show that two local moves are sufficient to generate the isotopy of rational tangles. These moves are called horizontal flypes and vertical flypes. They are basically a horizontal subtangle half twist and a vertical subtangle half twist. We just mention here that the isotopy of the closure arcs of a rational knot or link in the solid torus is represented by horizontal flypes, which amount to Reidemeister II moves.  

\begin{example}
The two rational links in the solid torus with rational tangles $%
\left[ -2\text{ }3\text{ }2\right] ,$ and $\left[ 3\text{ }-2\text{ }3\right]
$ are isotopy equivalent as their rational tangles have the same 
fraction $\dfrac{12}{5}$ (See Figure \ref{closure}).

\begin{figure}[H]
  \centering
   {\includegraphics[scale=0.1]{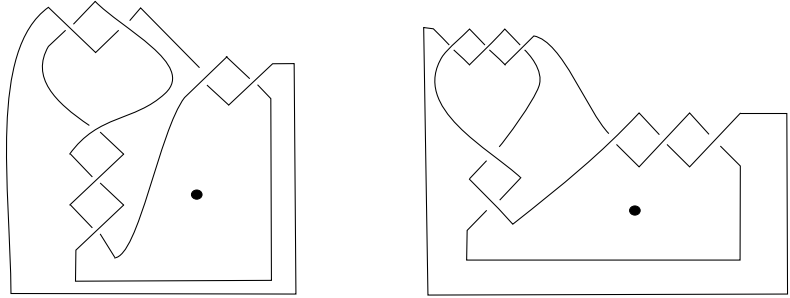}
    \caption{Two isotopy equivalent rational links in the solid torus.}
  \label{closure}}
\end{figure}

\end{example}

Recall that a rational tangle $T=[a_1,....a_n]$ is in a canonical  form  if $T$ is alternating and $n$ is odd. A result by Kauffman and Lambropoulou \cite{MR1953344} states that every rational tangle is equivalent to a rational tangle in a canonical form. We have the following definition.
\begin{definition}
A rational knot in the solid torus is said to be in a canonical form if it is the closure of a rational tangle in a canonical form .
\end{definition}
The following result follows from Lemma $3.5$ \cite{MR1953344}.
\begin{corollary}
Every rational knot or link in the solid torus can be isotoped to a canonical form.
\end{corollary}

The following definition is due to Kauffman and Lambropoulou \cite{dna}.
\begin{definition}
An unoriented rational tangle has connectivity type $[0]$ if the NW end arc is
connected to the NE end arc and the SW end arc is connected to the SE end
arc. These are the same connections as in the tangle $[0]$. Similarly, the
tangle has connectivity type [$\infty $] or $[1]$ if the end arc connections are the same as in the tangles [$\infty $] and $[+1]$ %(or equivalently $[-1]$) 
respectively.
\end{definition}

The basic connectivity patterns of rational tangles are denoted by the
tangles [0], [$\infty $] and [+1], and they can be represented iconically by%
\begin{eqnarray*}
\lbrack 0] &=&\asymp \\
\lbrack \infty ] &=&>< \\
\lbrack 1] &=&\lbrack -1]=\chi
\end{eqnarray*}

 The \textit{parity} \cite{dna} of a reduced fraction $p/q$ is defined to be the ratio
of the parities $(e$ $or$ $o)$ of its numerator and denominator $p$ and $q$,
where $e$ stands for even and $o$ for odd. Moreover, the tangle $[0]$ has
fraction $0 = 0/1$, thus parity $e/o$. The tangle [$\infty $] has formal
fraction $\infty  = 1/0$, thus parity $o/e$. The tangle $[+1]$ has fraction $1 =
1/1$, thus parity $o/o$, and the tangle $[-1]$ has fraction $-1 = -1/1$, thus
parity $o/o$. The connectivity type of a tangle is completely characterized by the parity of the tangle as the following theorem states. See \cite{dna}.

\begin{theorem}
\label{dna them}
A rational tangle $T$ has connectivity type $\asymp $ if and only if its fraction has parity $e/o$. The tangle $T$ has connectivity type $><$ 
if and only if its fraction has parity $o/e$. Finally, $T$ has connectivity type $\chi $ if and only if its fraction has parity $o/o$.
\end{theorem}
The previous theorem can be used to characterize the homotopy type of rational knots and links in the solid torus. Since connectivity does not change by Reidemeister moves and crossing changes we have the following.  
\begin{theorem}
\label{new them}
A rational link in the solid torus $\overline{R}$ is either
\begin{enumerate}

\item A two component link of homotopy type \begin{minipage}[h]{0.05\linewidth}
		\vspace{0pt}
		\scalebox{0.1}{\includegraphics{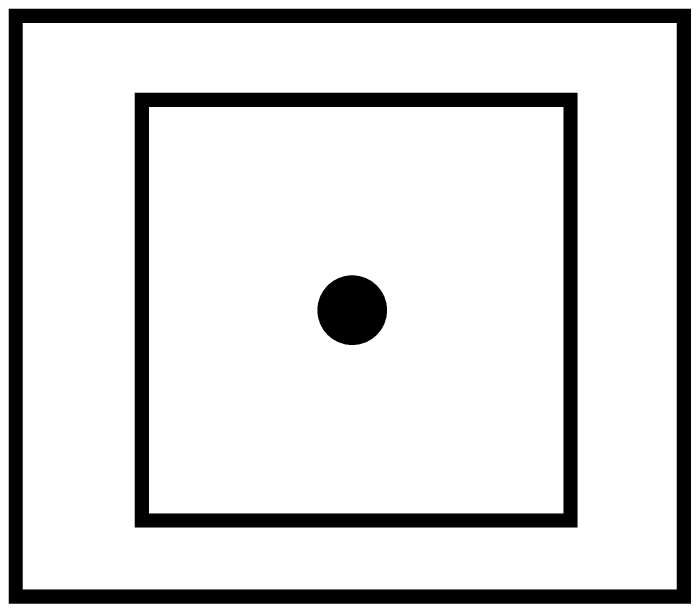}}\end{minipage}

 or

\item A knot of homotopy type \begin{minipage}[h]{0.05\linewidth}
		\vspace{0pt}
		\scalebox{0.1}{\includegraphics{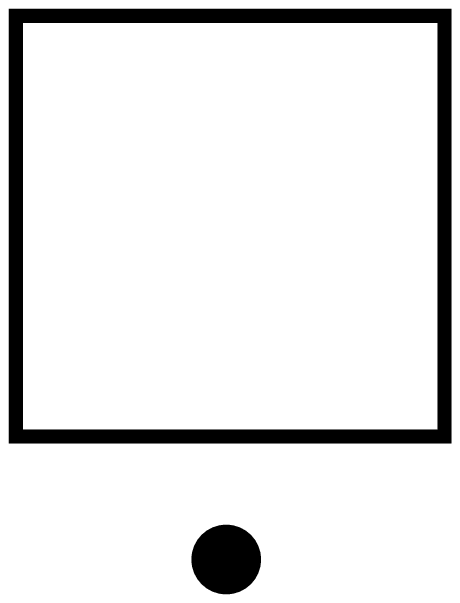}}\end{minipage}

or

\item A knot of homotopy type \begin{minipage}[h]{0.05\linewidth}
		\vspace{0pt}
		\scalebox{0.1}{\includegraphics{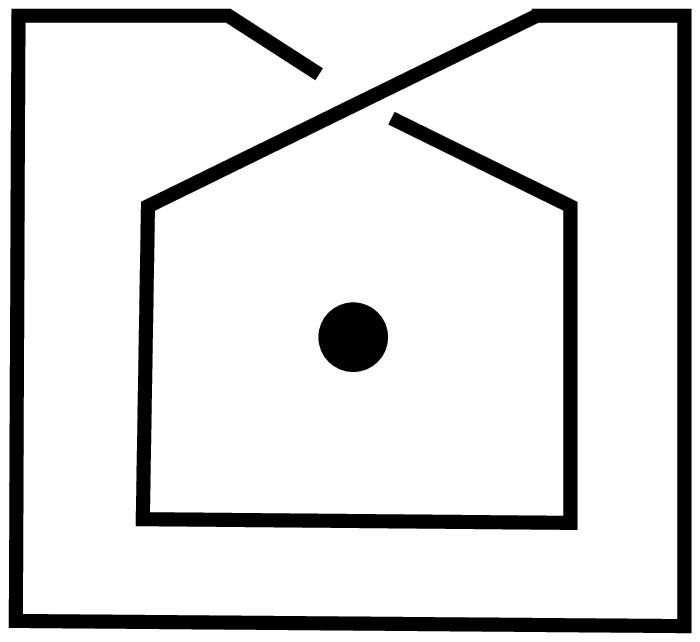}}\end{minipage}
\end{enumerate}

\end{theorem}

Theorems \ref{dna them} and \ref{new them} imply immediately the following.

\begin{corollary}
A rational link in the solid torus $\overline{R}$ is
\begin{enumerate}

\item  A two component link of homotopy type \begin{minipage}[h]{0.05\linewidth}
		\vspace{0pt}
		\scalebox{0.1}{\includegraphics{zz_ST}}\end{minipage} if and only if the fraction $F(R)$ has parity $o/e$.

\item  A knot of homotopy type \begin{minipage}[h]{0.05\linewidth}
		\vspace{0pt}
		\scalebox{0.1}{\includegraphics{trivial_ST}}\end{minipage} if and only the fraction $F(R)$ has parity $e/o$.

 \item  A knot of homotopy type \begin{minipage}[h]{0.05\linewidth}
		\vspace{0pt}
		\scalebox{0.1}{\includegraphics{crossing_ST}}\end{minipage} if and only the fraction $F(R)$ has parity $o/o$.

\end{enumerate}
\end{corollary}

\begin{remark}

One might suspect that the techniques provided in Theorem \ref{main theorem} can be extended for all $2$-tangles. Unfortunately, this is not the case. Here we provide a counter example.

\begin{figure}[h]
  \centering
   {\includegraphics[scale=0.09]{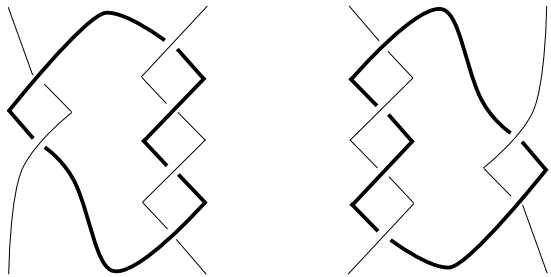}
    \caption{The tangle $T_1$ on the left and the tangle $T_2$ on the right.}
  \label{counter example}}
\end{figure}
Figure \ref{counter example} shows an example of two tangles $T_1$ and $T_2$ that are not isotopic. To show this we define a simple tangle invariant that is essentially similar to the classical linking number invariant of links in the 3-sphere. Let $T$ be a 2-tangle of the homotopy type given in Figure \ref{simple}.

\begin{figure}[h]
  \centering
   {\includegraphics[scale=0.15]{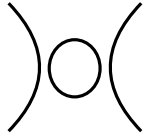}
    \caption{A simple 2-tangle.}
  \label{simple}}
\end{figure}
The invariant that we will define we will be defined on all 2-tangles $T$ that are homotopy equivalent to the 2-tangle in Figure \ref{simple}. We define the \textit{ left linking number} of $ T$ of the homotopy type in Figure \ref{simple} to be: one half of the absolute value of the sum of the signs between the left component and the component homotopic to the circle. Here by the left component we mean the component that connects NW to SW. This is clearly an invariant under Reidemeister moves (the proof of invariance is exactly the same as the proof in the case of the classical linking number on links in the 2-sphere). Note that the tangles $T_1$ and $T_2$ are both of the homotopy type given in Figure \ref{simple}.  Since the left linking number of $T_1$ is 1 and the left linking number of $T_2$ is $2$ we conclude that the tangles $T_1$ and $T_2$ are not isotopic.

On the other hand, the solid torus  closures of the tangles $T_1$ and $T_2$ are the links $\overline{T_1}$ and $\overline{T_2}$ shown in Figure \ref{14} are isotopic.

\begin{figure}[h]
  \centering
   {\includegraphics[scale=0.06]{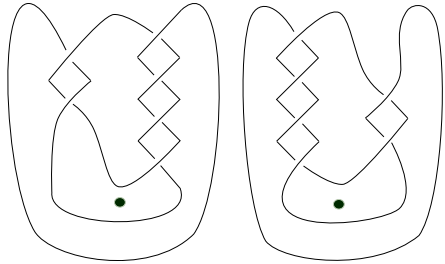}
    \caption{The two links $\overline{T_1}$ on the left and $\overline{T_2}$ on the right are isotopic.}
  \label{14}}
\end{figure}
 Indeed Figure \ref{15} shows the isotopy that relates the link $\overline{T_1}$ to the link $\overline{T_2}$.
\begin{figure}[h]
  \centering
   {\includegraphics[scale=0.05]{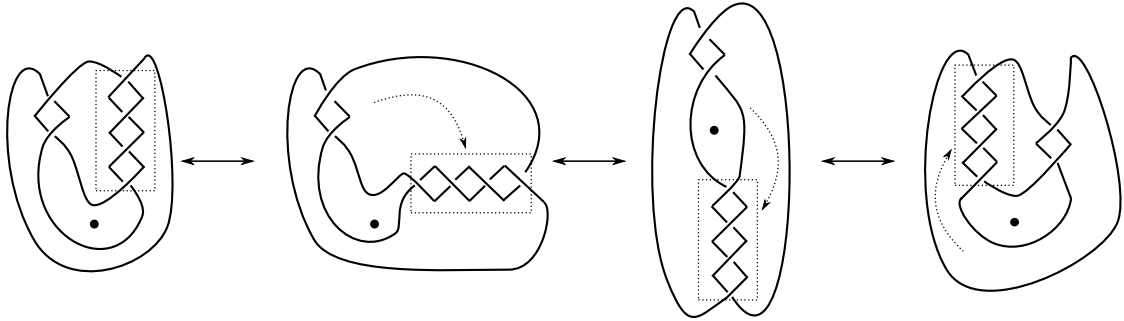}
    \caption{The link on the far left $\overline{T_1}$ is isotopic to the link of the far right $\overline{T_2}$ via the sequence of isotopies illustrated in the figure.}
  \label{15}}
\end{figure}
\end{remark}

In our definition of rational knots and links in the solid torus we did not include the classical rational knots and links that could be enclosed in a 3-ball embedded in the solid torus. These are exactly rational knots and links in the three space, and they were first classified in a theorem due to Schubert \cite{Ro}. Recall here that a numerator closure for a  rational tangle in $\mathbb{R}^3$ is the link in $\mathbb{R}^3$ obtained by joining the north endpoints together and the south endpoints also together.
\begin{theorem}(Schubert \cite{Ro})
Suppose that two rational tangles with fractions $(p/q)$ and $(p'/q')$ are given ($p$ and $q$ are relatively prime. Similarly for $p'$ and $q'$.) If $K(p/q)$ and $K(p'/q')$ denote the corresponding rational knots obtained by taking numerator closures of these tangles, then $K(p/q)$ and $K(p'/q')$ are isotopic if and only if $p=p'$ and either $q=q' \mod p$ or $qq'=1 \mod p$.
\end{theorem}

\section{Generalizations}
\label{Generalizations}
In this section we assume that the reader is familiar with skein theory associated with quantum spin networks \cite{masbum} and their connections with Kauffman bracket skein modules. In particular we assume familiarity with the graphical definition of the Jones-Wenzl projector, its properties and its connection with trivalent graphs. For more details see  \cite{Lic92}.\\

The main ingredient of Theorem \ref{main theorem} is the fact that the map $STC_2:TL_2 \longrightarrow \mathcal{S}(S^1\times I)$ is an isomorphism on its image. This follows from the fact that both modules $TL_2$ and $\Ima STC_2$ are free modules of rank $2$. In this section we give a generalization for this map and for Theorem \ref{Kauffman Goldman}. 

Consider the restriction of the map (\ref{src}) $STC_{2n}:TL_{2n} \longrightarrow \mathcal{S}(S^1\times I)$ on $TL^n_4$. We will denote this restriction by $STC^{\prime}_{2n}$. We have the following theorem.
\begin{theorem}
\label{thm}
The map $STC^{\prime}_{2n}:TL_4^{n} \longrightarrow \mathcal{S}(S^1\times I) $ is a $\mathbb{Q}(A)$-module isomorphism on its image.
\end{theorem}
\begin{proof}
The fact that the module $TL^n_4$ is $(n+1)$-dimensional generated by the basis $\{B_{n,i}\}_{i=0}^n$ given in Figure \ref{basis111} can be found in \cite{Lickbook}.
      
\begin{figure}[h]
  \centering
   {\includegraphics[scale=0.33]{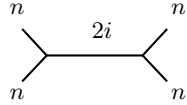}
    \put(-29,+17){\footnotesize{$2i$}}
        \put(-60,+26){\footnotesize{$n$}}
        \put(-60,-6){\footnotesize{$n$}}
        \put(1,-6){\footnotesize{$n$}}
        \put(1,+26){\footnotesize{$n$}}
    \caption{The basis element $B_{n,i}$ in module $\mathcal{D}^4_n$.}
  \label{basis111}}
\end{figure}
Now consider the image of the basis element $B_{n,i}$ under the map $STC^{\prime}_{2n}$:

\begin{figure}[H]
  \centering
   {\includegraphics[scale=0.5]{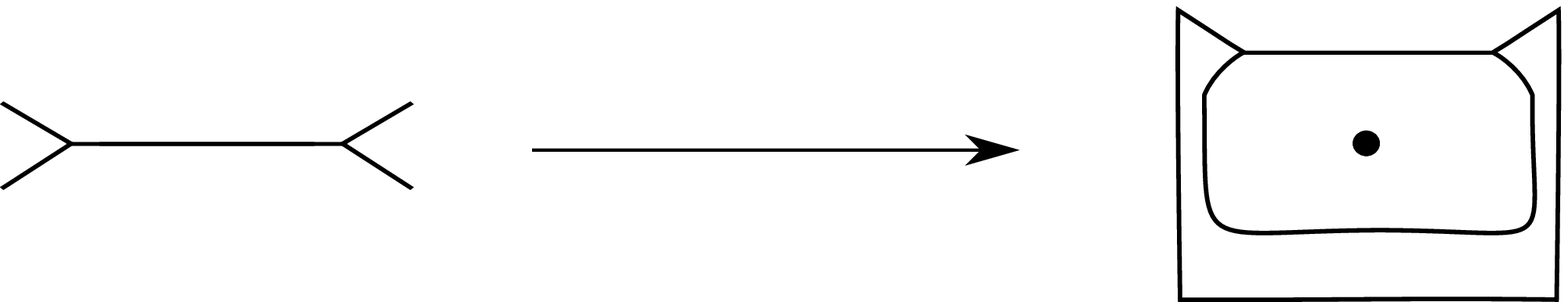}
    \put(-29,+37){\footnotesize{$2i$}}
        \put(-60,+26){\footnotesize{$n$}}
        \put(-60,-6){\footnotesize{$n$}}
         \put(-257,+32){\footnotesize{$2i$}}
        \put(-283,+36){\footnotesize{$n$}}
        \put(-283,13){\footnotesize{$n$}}
        \put(-223,+36){\footnotesize{$n$}}
        \put(-223,13){\footnotesize{$n$}}
        \put(-155,+35){\footnotesize{$STC^{\prime}_{2n}$}}
      \label{closures}}
\end{figure}
On the other hand, one has the identity:

\begin{eqnarray*}
   \begin{minipage}[h]{0.07\linewidth}
         \vspace{0pt}
         \scalebox{0.29}{\includegraphics{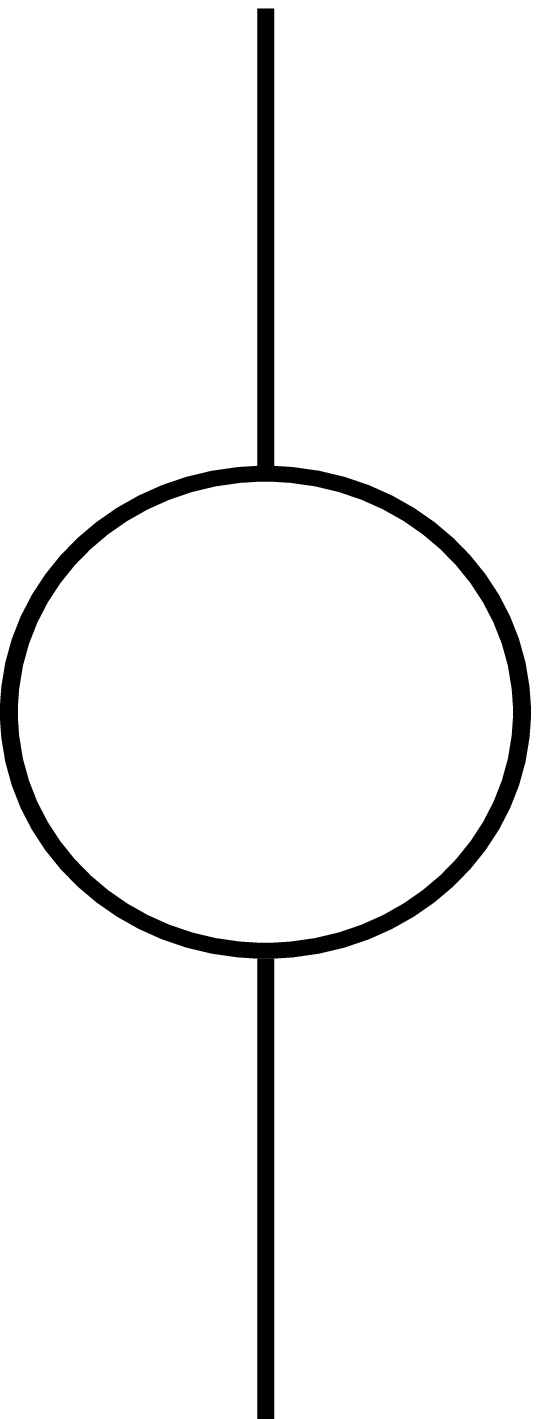}}
          \put(-35,100){$2i$}
          \put(-35,5){$2i$}
         \put(-41,60){$n$}
         \put(-11,60){$n$}
   \end{minipage}\quad&=& \frac{\theta(n,n,2i)}{\Delta_{2i}} \quad
   \begin{minipage}[h]{0.08\linewidth}
         \vspace{10pt}
         \scalebox{0.29}{\includegraphics{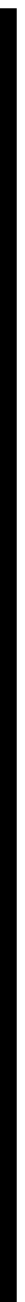}}
        \put(3,100){$2i$}         
   \end{minipage}
  \end{eqnarray*}
\vspace{10pt}  
The previous identity and the coefficients $\theta(n,n,2i)$ and $\Delta_{2i}$ are defined in \cite{masbum}. Hence,
\begin{eqnarray*}
STC^{\prime}_{2n}(B_{n,i})
   \quad= \quad
   \begin{minipage}[h]{0.08\linewidth}
         \vspace{-10pt}
         \scalebox{0.39}{\includegraphics{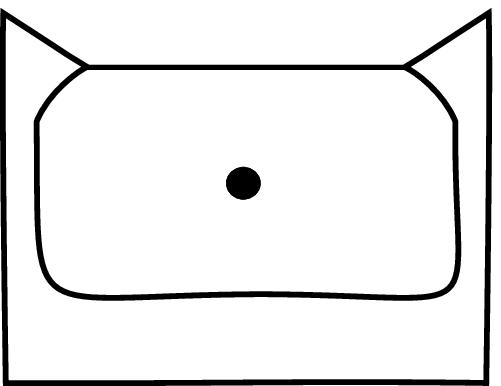}}
        \put(-30,40){$2i$}
        \put(-25,12){$n$}
        \put(-45,2){$n$}
   \end{minipage}\quad \quad= \frac{\theta(n,n,2i)}{\Delta_{2i}}\quad  \begin{minipage}[h]{0.08\linewidth}
         \vspace{-10pt}
         \scalebox{0.39}{\includegraphics{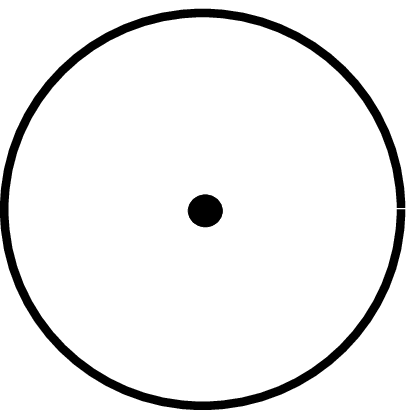}}
        \put(-30,50){$2i$}
   \end{minipage}
  \end{eqnarray*}
On the other hand the following holds:
\begin{eqnarray*}
  \begin{minipage}[h]{0.08\linewidth}
         \vspace{-10pt}
         \scalebox{0.39}{\includegraphics{Si}}
        \put(-30,50){$2i$}
   \end{minipage}\quad = S_{2i}
  \end{eqnarray*}
where $S_n$ is the $n^{th}$ Chebyshev polynomial \cite{Lickbook}. Since the ratio $\frac{\theta(n,n,2i)}{\Delta_{2i}}$ is not zero and $\{S_i\}_{i=0}^{\infty}$ is a basis for the module $\mathcal{S}(S^1\times I)$ \cite{Lickbook}, the result follows. 
\end{proof}

Let $T$ be a rational tangle. Color all connected components of $T$ by the $n^{th}$ Jones-Wenzl projector and denote this colored tangle by $T^n$.  We will call $T^n$ the $n^{th}$ colored tangle of $T$. Consider the expansion of $T^n$ in terms of the basis $B_{n,i}$ defined in Theorem \ref{thm} :
\begin{equation}
\label{general}
T^n= \sum_{i=0}^n \gamma_i(T)B_{n,i}
\end{equation}
We then have the following generalization of Theorem \ref{Kauffman Goldman}:

\begin{theorem}
\label{colored tangles}
Let $T$ be a rational tangle and $T^{n}$ be its $n^{th}$ colored tangle. Then for any choice of $A$ and for $0\leq i \leq n-1$  the ratios:
\begin{equation}
CR^i_T(A)=\frac{\gamma_i(T)}{\gamma_{n}(T)} 
\end{equation}
are ambient isotopy invariants of the tangle $T$, where the  $\gamma_i(T)$'s are the coefficients defined in equation \ref{general}. 
\end{theorem}

\begin{proof}
The fact that $\{B_{n,i}\}_{i=0}^n$ is a basis for $TL^n_4$ implies that  $\{\gamma_i(T)\}_{i=0}^{n}$ are invariants under type-$II$ and type-$III$ moves. If $T_1$ is a tangle obtained from $T$ by doing a type $I$-move then we can see from equation \ref{thisthis} that $T_{1}^n = \mu_n^{\pm 1} T^n $. In other words every element in $\{\gamma_i(T)\}_{i=0}^{n}$ is multiplied by the same factor $\mu_n^{\pm 1}$ and hence the ratios $\frac{\gamma_i(T)}{\gamma_{n}(T)}$  do not change under type-$I$ move. The result follows.
\end{proof}
\begin{remark}
The basis $\{B_{n,i}\}_{i=0}^n$ in Theorem \ref{colored tangles} does not specialize to the basis $[\infty]$ and $[0]$ in Theorem \ref{Kauffman Goldman}. In fact, the choice of basis is not essential in the proof of both theorems and one could choose any basis. In particular the basis of $TL_{4}^n$ that specializes to $\{[\infty],[0]\}$ is given in Figure \ref{222}

\begin{figure}[H]
  \centering
   \scalebox{.4}{\includegraphics{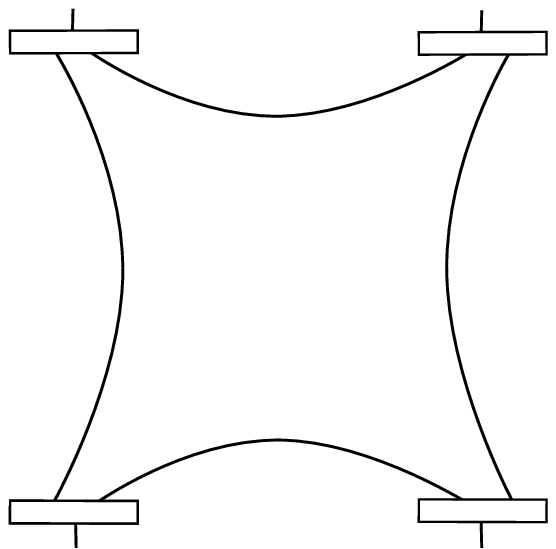}}
        \scriptsize{
        \put(-60,+64){$n$}
          \put(-5,+64){$n$}
          \put(-58,30){$k$}
          \put(-7,30){$k$}
         \put(-43,+56){$n-k$}
           \put(-43,+16){$n-k$}
           }
           \caption{For $n=1$ this basis element specializes to the basis consisting of the two elements $[\infty]$ and $[0]$.}
           \label{222}
\end{figure}          
          
\end{remark}

Note that the assumption of $T$ being  a rational tangle is not necessary in the statement of the previous theorem. Theorem \ref{colored tangles} is still valid for any $2$-tangle.

\begin{remark}
If the link $L$ in the solid torus is obtained as a closure of the $2$-tangle $T$, then the coefficients $\Gamma_i(L)$ in Theorem \ref{solidtorus} are determined uniquely and completely by the coefficients $\gamma_i(T)$ in Theorem \ref{colored tangles} and vice-versa. This follows immediately from Theorem \ref{thm}.      
\end{remark}
The previous observations carry over to other Kauffman bracket skein modules. More precisely, let $L$ be a link in the solid torus. Let $L^{n}$ be the skein element obtained from $L$ by decorating all its elements by the $n^{th}$ Jones-Wenzl projector. We have the following result which generalizes parts of Theorem \ref{main theorem}: 

\begin{theorem}
\label{solidtorus}
Let $L$ be a link in the solid torus and consider the expansion of $L^{n}$:
\begin{equation}
L^{n}= \sum_{i=0}^{k} \Gamma_{i}(L) S_i  
\end{equation}
then the ratios $\frac{\Gamma_{i}}{\Gamma_{k}}$, where $0\leq  i \leq k-1$ , are ambient isotopy invariants for the link $L$.
\end{theorem}
Theorems \ref{colored tangles} and \ref{solidtorus} can be generalized as follows.  

\begin{theorem}
\label{tangle general}
Let $T$ be a $(k,l)$ tangle and let $T^{n}$ be its $n^{th}$ colored tangle. Let $\{E_i\}_{i=0}^{s}$ be a basis for the module $\mathcal{D}^{k}_{l}(n)$, then for $0 \leq i \leq s-1 $ the ratios
\begin{equation}
CR^i_T(A)=\frac{\gamma_i(T)}{\gamma_{s}(T)} 
\end{equation} 
are ambient isotopy invariants of the tangle $T$. 
\end{theorem}
Finally, one has the following result that generalizes Theorems \ref{tangle general} and \ref{colored tangles}.

\begin{theorem}
\label{main theorem2} Let $D$ be a diagram in a surface $F$ with $m$ marked points. Let $D^{n}$ be the diagram obtained from $D$ by coloring all its components with the $n^{th}$ Jones-Wenzl projector. Let $\mathcal{E}= \{E_i\}_{i=0}^s$ be a basis for the skein module $\mathcal{S}(F_{m}(n))$ and let $\gamma_i(D)$ be the $i^{th}$ coefficient obtained from expanding $D^n$ in terms of the basis $\mathcal{E}$. Then for any $ 0\leq i \leq s-1 $ the ratios:
\begin{equation}
CR^i_D(A)=\frac{\gamma_i(D)}{\gamma_{s}(D)} 
\end{equation} 
are ambient isotopy invariants of the  diagram $D$.
\end{theorem}

\end{document}